\documentclass[a4paper]{amsart}
\usepackage[english]{babel}
\usepackage{amssymb,amsmath,amsthm}
\usepackage{mathrsfs}
\usepackage{enumitem}

\usepackage{hyperref}

\theoremstyle{plain}
\newtheorem{teo}{Theorem}[section]
\newtheorem{lem}[teo]{Lemma}
\newtheorem{coro}[teo]{Corollary}
\newtheorem{prop}[teo]{Proposition}
\theoremstyle{definition}
\newtheorem{defi}[teo]{Definition}
\newtheorem{rem}[teo]{Remark}
\newtheorem*{assm}{Assumptions}
\numberwithin{equation}{section}

\newcommand{\NN}{\mathbb{N}}
\newcommand{\ZZ}{\mathbb{Z}}
\newcommand{\RR}{\mathbb{R}}
\newcommand{\CC}{\mathbb{C}}

\newcommand{\Rpos}{{\RR^+}}
\newcommand{\Rnon}{{\RR^+_0}}

\newcommand{\di}{\,\mathrm{d}}    
\newcommand{\ndi}{\mathrm{d}}    

\newcommand{\vfG}{X}               
\newcommand{\vfA}{\breve{X}_0}     
\newcommand{\vfN}{\breve{X}}       

\newcommand{\step}{S}

\newcommand{\hardy}{\mathfrak{h}^1}   
\newcommand{\bmo}{\mathfrak{bmo}}     

\newcommand{\AH}{\mathbf{H}}        
\newcommand{\AcH}{\mathscr{H}}      
\newcommand{\AC}{\mathbf{C}}        
\newcommand{\AW}{\mathbf{W}}        

\newcommand{\DD}{\mathrm{D}}      
\newcommand{\Dist}{\mathcal{D}}   

\newcommand{\lie}{\mathfrak}
\DeclareMathOperator{\tr}{tr}
\DeclareMathOperator{\Aut}{Aut}

\DeclareMathOperator{\sgn}{sgn}

\DeclareMathOperator{\Span}{span}

\DeclareMathOperator{\supp}{supp}

\newcommand{\dist}{\varrho}        

\newcommand{\tc}{\,:\,}

\newcommand{\loc}{\mathrm{loc}}
\newcommand{\sloc}{\mathrm{sloc}}

\newcommand{\Balls}{\mathcal{B}}   

\newcommand{\thresh}{\varsigma}    
\newcommand{\lebexp}{\wp}          
\newcommand{\plnexp}{\varkappa}    

\newcommand{\acknowledgments}{\section*{Acknowledgments} This work was partially supported by the EPSRC Grant ``Sub-Elliptic Harmonic Analysis'' (EP/P002447/1), the Progetto GNAMPA 2017 ``Analisi armonica e teoria spettrale di Laplaciani", the Progetto GNAMPA 2016 ``Calcolo funzionale per operatori subellittici su variet\`a'',  
and the Progetto PRIN 2015 ``Variet\`a reali e complesse: geometria, topologia e analisi armonica''. Martini and Vallarino are members of the Gruppo Nazionale per l'Analisi Matematica, la Probabilit\`a e le loro Applicazioni (GNAMPA) of the Istituto Nazionale di Alta Matematica (INdAM). }

\begin{document}

\title
[A multiplier theorem for sub-Laplacians with drift]
{A multiplier theorem for sub-Laplacians with drift on Lie groups}

\author[A. Martini]{Alessio Martini}
\address[A. Martini]{School of Mathematics \\ University of Birmingham \\
Edgbaston \\ Birmingham \\ B15 2TT \\ United Kingdom}
\email{a.martini@bham.ac.uk}

\author[A. Ottazzi]{Alessandro Ottazzi}
\address[A. Ottazzi]{School of Mathematics and Statistics \\ University of New South Wales \\ UNSW Sydney NSW 2052 \\ Australia}
\email{a.ottazzi@unsw.edu.au}

\author[M. Vallarino]{Maria Vallarino}
\address[M. Vallarino]{Dipartimento di Scienze Matematiche ``Giuseppe Luigi Lagrange''
\\ Politecnico di Torino\\
Corso Duca degli Abruzzi 24\\ 10129 Torino\\ Italy}
\email{maria.vallarino@polito.it}

\subjclass[2010]{22E30, 42B15, 42B20, 43A22}
\keywords{spectral multiplier, sub-Laplacian, Lie group, drift, Hardy space}

\begin{abstract}
We prove a general multiplier theorem for 
symmetric left-invariant sub-Laplacians with drift on non-compact Lie groups.
This considerably improves and extends a result by Hebisch, Mauceri, and Meda.
Applications include groups of polynomial growth and solvable extensions of stratified groups.
\end{abstract}

\maketitle

\section{Introduction}

Let $G$ be a connected Lie group.
 Let $X_1,\dots, X_{\nu}$ be
 left-invariant vector fields on $G$ that satisfy H\"ormander's condition
and $\dist$ be the Carnot--Carath\'eodory distance associated with $X_1,\dots, X_{\nu}$.
Then, as it is well known, the sub-Laplacian $\Delta=-\sum_{j=1}^{\nu}X_j^2$ is hypoelliptic and essentially self-adjoint on $L^2(\mu)$, where $\mu$ is the right Haar measure of $G$.  

Let $\chi$ be a nontrivial positive character of $G$; note that the existence of $\chi$ forces $G$ to be non-compact. Define the vector field
\begin{equation}\label{eq:drift}
X=\sum_{j=1}^{\nu}\di\chi_e(X_j)X_j,
\end{equation}
where $e$ denotes the identity of $G$. Then the sub-Laplacian with drift $\Delta_X = \Delta - X$ is essentially selfadjont on $L^2(\mu_X)$, where $\di\mu_X=\chi\di\mu$, and its $L^2(\mu_X)$-spectrum is contained
 in $[b_X^2,\infty)$, where $b_X = (\sum_j |\ndi\chi_e(X_j)|^2/4)^{1/2}$. The above form of the drift $X$ is not an arbitrary choice: indeed, as shown in \cite{HMM}, any left-invariant vector field $X$ on $G$ such that $\Delta-X$ is symmetric on $L^2(\tilde\mu)$ for some positive measure $\tilde\mu$ on $G$ has the form \eqref{eq:drift} for some positive character $\chi$ of $G$.

Such sub-Laplacians with drift have been extensively studied in the literature.
When $G$ is a nonunimodular Lie group with a left-invariant sub-Riemannian structure, the ``intrinsic hypoelliptic Laplacian'' considered in \cite{ABGR} is a sub-Laplacian with drift as defined above; this includes, in particular, the Laplace--Beltrami operator associated with any left-invariant Riemannian metric on $G$.
Heat kernel estimates for sub-Laplacians with drift were studied on various Lie groups in \cite{A2, D, varopoulos_analysis_1992}. Lohou\'e and Mustapha \cite{LM} studied the $L^p$ boundedness of the Riesz transforms of any order associated with sub-Laplacians with drift on every amenable Lie group; endpoint estimates for the Riesz transforms of any order associated with the Laplacian with drift were proved in \cite{LS,LSW} in the case where $G=\RR^n$.

In this paper we are interested in $L^p$ spectral multipliers of $\Delta_X$, i.e., the bounded Borel functions $M : [b_X^2,\infty) \to \CC$ such that the operator $M(\Delta_X)$, initially defined on $L^2(\mu_X)$, extends to a bounded operator on $L^p(\mu_X)$. Hebisch, Mauceri and Meda \cite{HMM} showed that, if $G$ is amenable and $p \in [1,\infty] \setminus \{2\}$, then every $L^p$ spectral multiplier of $\Delta_X$ extends to a bounded holomorphic function on a parabolic region $P_{X,p}$ in the complex plane depending on $p$ and on the drift $X$; namely,
\begin{equation}\label{eq:parabola}
P_{X, p}=\left\{x+iy\in\CC \tc x > \frac{y^2}{4b_X^2\,\sin ^2\phi_p^*}+ b_X^2\,\cos ^2\phi_p^* \right\} ,
\end{equation}
where $\phi_p^*=\arcsin |2/p-1|$. In the case where $G$ has polynomial growth and $1 < p < \infty$, they also found a sufficient condition, stated in \cite[Theorem 5.2]{HMM}, for a holomorphic function $M$ on $P_{X,p}$ to be an $L^p$ spectral multiplier of $\Delta_X$. This condition is more conveniently expressed by means of the change of variable
\begin{equation}\label{eq:change_var}
M_X(z) = M(b_X^2+z^2),
\end{equation}
which defines a holomorphic function $M_X$ on the strip
$\Sigma_{W} = \{ x+iy \in \CC \tc |y| < W \}$
of half-width $W = |2/p-1| \, b_X$; 
with this notation, the sufficient condition in \cite{HMM} takes the form
\begin{equation}\label{eq:nhmh}
\max_{k \in \{0,\dots,N\}} \sup_{z \in \overline{\Sigma_{W}}} (1+|z|)^k |M_X^{(k)}(z)| < \infty,
\end{equation}
where the number $N$ of derivatives to be controlled depends on the group $G$ and the sub-Laplacian $\Delta$, but is independent of $p$.

The main result of this paper improves and complements \cite[Theorem 5.2]{HMM} in several ways. First of all, we refine the sufficient condition for $L^p$-boundedness, by reducing the number $N$ of derivatives and making it dependent on $p$. For instance, in the case where $G = \RR^n$ and $\Delta$ is the standard Laplacian, \cite[Theorem 5.2]{HMM} requires $N > (n+4)/2$, while our result requires $N > |1/p-1/2|(n+1)$. Actually, in the sharpest formulation of our result, the pointwise differential condition \eqref{eq:nhmh} of integer order $N$ is replaced by an $L^q$ condition of fractional order.

Secondly, we obtain an endpoint result in the cases $p=1$ and $p=\infty$, in terms of suitable Hardy and BMO spaces adapted to the measured metric space $(G,\dist,\mu_X)$. In this we exploit the general theory of Hardy and BMO spaces of Goldberg type on measured metric spaces satisfying mild geometric conditions that has recently been developed by Meda and Volpi \cite{MV}.

Thirdly, we extend the range of applicability beyond the class of Lie groups of polynomial growth: namely, our result applies also to distinguished sub-Laplacians on some Lie groups of exponential growth, such as the rank-one solvable extensions of stratified groups considered in \cite{HS,MOV}.

In order to present our multiplier theorem in full generality, we state it as a conditional result (see Theorem \ref{teo:main} and Corollary \ref{cor:main} below). In other words, we prove a spectral multiplier theorem for a sub-Laplacian with drift $\Delta_X$ on an arbitrary group $G$, provided that certain estimates (stated as Assumptions \ref{en:ass_MH_diff}, \ref{en:ass_MH_annulus} and \ref{en:ass_fps_plancherel} in Section \ref{s:main} below) hold. These assumptions are technical in nature, however it turns out that they can be verified in many cases, provided that the sub-Laplacian without drift $\Delta$ has a differentiable $L^p(\mu)$ functional calculus and a multiplier theorem of Mihlin--H\"ormander type for $\Delta$ holds.

Next we show how the general multiplier result can be applied in particular cases. Theorem \ref{teo:main-pol} deals with the case where $G$ has polynomial growth (in this case a multiplier theorem for $\Delta$ was obtained in \cite{A,cowling_spectral_2001,DOS}). Theorem \ref{teo:main_na}, instead, considers the case where $\Delta$ is a distinguished sub-Laplacian on a rank-one solvable extension of a stratified group (a multiplier theorem for such $\Delta$ was obtained in \cite{HS,MOV}). These are but a few examples of applications of the general conditional result, which could also be used, for instance, to refine the multiplier theorem for a complete Laplacian with drift on Damek--Ricci spaces obtained in \cite{OV} (see also \cite{V2} for the case of the Laplacian without drift).

Note that the parabolic region $P_{X, p}$ defined in \eqref{eq:parabola} is tangent to the sector of angle $\phi_p^*$ and vertex $0$ in the complex plane, and Carbonaro and Dragi\v cevi\' c proved  \cite{CD} that every generator of a symmetric contraction semigroup has holomorphic functional calculus on that sector. Therefore our result, when applicable, provides a more precise
 spectral multiplier theorem for $\Delta_X$.

In the case where $\Delta_X$ is the Laplace--Beltrami operator associated with a left-invariant Riemannian metric on $G$, an $L^p$ spectral multiplier theorem for $\Delta_X$ is included in the general result of Taylor \cite{taylor} (see also \cite{CGT}) for the Laplace--Beltrami operator on complete Riemannian manifolds with bounded geometry. The differential condition on the multiplier in Taylor's result has the form \eqref{eq:nhmh}, but the number of derivatives to be controlled is not specified. In addition, that result only applies  to elliptic operators, while the sub-Laplacians considered here need not be elliptic.

In the particular case where $G$ is a rank-one Riemannian symmetric space of the non-compact type and $\Delta_X$ is the corresponding Laplace--Beltrami operator, our multiplier theorem reduces, up to the endpoint, to a particular instance of that of Anker \cite{anker} for spherical Fourier multipliers (see also \cite{anker_multiplicateurs_1986,clerc_lp_1974,cowling_estimates_1993,giulini_lp_1997,stanton_expansions_1978} for related results). Indeed, the interpolation machinery that we exploit here to obtain a multiplier theorem with a $p$-dependent smoothness condition was essentially developed in \cite{anker}. On the other hand, in the proof of the multiplier theorem of \cite{anker}, a substantial role is played by spherical analysis and properties of the Abel transform, which are not available in our general context.

In particular, the ``endpoint result'' for $p=1$ in \cite{anker} (from which the result for $1 < p < \infty$ follows by interpolation) involves differential conditions of $L^2$ type on the multiplier, corresponding to the fact that an $L^2$ Plancherel formula for the spherical transform is available. In our generality (and especially for an arbitrary group $G$ of polynomial growth) such a precise identity need not be available, and one may have to content oneself with rougher ``Plancherel estimates'' of $L^\infty$ type.

For this reason, here we develop a precise $L^q$ Paley--Wiener theory for $q \in (1,\infty)$ for holomorphic functions on a strip (see Lemma \ref{lem:paleywiener} below), which is parallel to the Plancherel-based $L^2$ theory exploited in \cite{anker}; when applied with $q$ arbitrarily large, this allows us to avoid a ``loss of derivatives'' and obtain a more precise result than in \cite{HMM}. A further improvement derives from a finer analysis of the volume growth of balls and the observation that, in the case of polynomial growth, the measure of a sphere grows slower than the measure of a ball (see \eqref{eq:poly_sphere_bound} below).

The proof of our endpoint result is based on splitting the convolution kernel of the operator $M(\Delta_X)$ into a local and a global part. Similarly as in \cite{anker,HMM}, under a condition of the form \eqref{eq:nhmh} on the multiplier $M$, the local part of the kernel is then shown to satisfy estimates of Calder\'on--Zygmund type, while the global part is proved to be integrable. The theory of local Hardy spaces developed in \cite{MV} is perfectly suited to treat kernels of this kind. A challenging problem would be to investigate weaker versions of \eqref{eq:nhmh} that do not force the global part of the kernel to be integrable; in the case of Riemannian symmetric spaces of the non-compact type and spherical multipliers, results in this direction can be found in \cite{ionescu_singular_2002,ionescu_singular_2003,MMV_preprint,meda_weak_2010}.

\medskip
 
A few words about notation are in order. The letter $C$ and variants such as $C_s$ denote finite positive constants that may vary from place to place.
Given two expressions $A$ and $B$, $A\lesssim B$ means that there exists a finite positive constant $C$ such that $ A\le C \,B $. Moreover $A \sim B$ means $A \lesssim B$ and $B \lesssim A$.
For two subsets $U$ and $V$ of a topological space, we write $U \Subset V$ to denote that the closure of $U$ is compact and contained in $V$.

\section{Preliminaries}
\subsection{Sub-Laplacians with drift and their geometry}\label{s:preliminaries}
In this section we discuss some geometric properties of the space $(G,\dist, \mu_{X})$ defined in the Introduction and introduce the Hardy and BMO spaces which are used in our results.

We briefly recall the definition of the sub-Riemannian structure associated to a H\"ormander system of vector fields. Let
 $X_1,\dots,X_\nu$ be linearly independent left-invariant vector fields on $G$ satisfying H\"ormander's condition. The associated horizontal distribution is the left-invariant subbundle $HG$ of the tangent bundle $TG$ of $G$ defined by $H_x G = \Span\{X_1|_x,\dots,X_\nu|_x\}$ for all $x \in G$. A left-invariant inner product $\langle \cdot,\cdot \rangle$ on the fibres of $HG$ is defined by requiring that $X_1|_x,\dots,X_\nu|_x$ are orthonormal for all $x \in G$;
the corresponding norm is denoted by $|\cdot|$. The horizontal gradient $\nabla_H f$ of a (smooth) function $f : G \to \RR$ is the section of $HG$ defined by
\[
\nabla_H f = \sum_{j=1}^\nu (X_j f) X_j;
\]
in particular
\[
|\nabla_H f|^2 = \sum_{j=1}^\nu |X_j f|^2
\]
at each point of $G$. A horizontal curve in $G$ is an (absolutely continuous) curve $\gamma : I \to \RR$ such that $\gamma'(t) \in H_{\gamma(t)} G$ for almost every $t \in I$; its length is defined by $L(\gamma) = \int_I |\gamma'(t)| \di t$. The sub-Riemannian (or Carnot--Carath\'eodory) distance $\dist(x,y)$ between two points $x,y \in G$ is defined as the infimum of the lengths of all horizontal curves joining $x$ to $y$.

We denote by $B_\dist(x,r)$ 
the 
closed ball relative to $\dist$ of centre $x \in G$ and radius $r > 0$, i.e., $B_\dist(x,r) = \{ y \in G \tc \dist(y,x) \leq r\}$. We also write $|x|_\dist = \dist(x,e)$ for all $x \in G$. Let $\mu$ be a right Haar measure on $G$ and $m$ be the modular function of $G$; then $\mu_\ell = m \mu$ is a left Haar measure. Define
\[
V_\dist(r) = \mu(B_\dist(e,r)) = \mu_\ell(B_\dist(e,r))
\]
for all $r \in (0,\infty)$; the latter equality is due to the fact that
\[
|x^{-1}|_\dist = |x|_\dist
\]
for all $x \in G$. The following statement collects several well-known facts about left-invariant sub-Riemannian structures.

\begin{prop}\label{prp:basic_sub_riemannian}
The sub-Riemannian distance $\dist$ is finite, left-invariant and compatible with the topology of $G$. The metric space $(G,\dist)$ is a locally compact, complete length space. Moreover 
\[
V_\dist(r) \sim r^{d_0} \qquad\text{for all $r \in (0,1]$}
\]
for some $d_0 \in \NN$, $d_0 \geq \dim G$.
Further, if $G$ has exponential growth, then
\[
e^{ar} \lesssim V_\dist(r) \lesssim e^{br} \qquad\text{for all $r \in [1,\infty]$}
\]
for some $a,b \in (0,\infty)$; if instead $G$ has polynomial growth, then
\begin{equation}\label{eq:poly_infty_asymp}
V_\dist(r) \sim r^{d_\infty} \qquad\text{for all $r \in [1,\infty)$,}
\end{equation}
where $d_\infty \in \NN$ is the degree of polynomial growth of $G$, and there exists $\delta \in (0,1]$ such that
\begin{equation}\label{eq:poly_sphere_bound}
V_\dist(r+1)-V_\dist(r) \lesssim r^{-\delta} V_\dist(r)
\end{equation}
for all $r \in [1,\infty)$.
\end{prop}
\begin{proof}
The left-invariance of $\dist$ is an immediate consequence of the definition and the left-invariance of $X_1,\dots,X_\nu$. Finiteness and compatibility with the topology of $G$ are consequences of the connectedness of $G$ and H\"ormander's condition, by the Chow--Rashevsky and ball-box theorems \cite{Montgomery}. In particular $(G,\dist)$ is a locally compact length space (see, e.g., \cite[Chapter 2]{BBI}), which is complete due to left-invariance (see, e.g., \cite[\S III.3.3]{Bourbaki}). The behaviour of $V_\dist(r)$ for small $r$ is also a consequence of the ball-box theorem. About the relation of the growth of $V_\dist(r)$ for large $r$ with the intrinsic growth properties of the Lie group $G$, we refer to the discussion in \cite[\S III.4]{varopoulos_analysis_1992} and \cite{guivarch_croissance_1973} (see also 
the precise asymptotic of \cite[Corollary 1.6]{breuillard} in the case $G$ has polynomial growth). Finally, for the estimate \eqref{eq:poly_sphere_bound} we refer to \cite[Lemma 3.3]{coldingminicozzi} and \cite[Theorem 4]{tessera}.
\end{proof}

The choice of the measure $\mu$ on $G$ determines the identification of the space of locally integrable functions $L^1_\loc(\mu)$ on $G$ with a subspace of the space of distributions $\Dist'(G)$. By the Schwartz kernel theorem, all bounded operators $T : C^\infty_c(G) \to \Dist'(G)$ have an integral kernel $K_T^\mu \in \Dist'(G \times G)$, such that
\[
Tf(x) = \int_G K_T^\mu(x,y) \, f(y) \di \mu(y)
\]
in the sense of distributions. If $\tilde \mu = \phi \mu$ is another measure on $G$, with a smooth positive density $\phi$ with respect to $\mu$, then we can also consider the integral kernel $K_T^{\tilde\mu}$ of $T$ with respect to $\tilde\mu$, such that
\[
Tf(x) = \int_G K_T^{\tilde\mu}(x,y) \, f(y) \di \tilde\mu(y),
\]
where the two integral kernels are related by
\begin{equation}\label{eq:int_kernel_meas}
K_T^{\tilde\mu}(x,y) =  K_T^\mu(x,y) / \phi(y).
\end{equation}
Further, if $T$ is left-invariant, then it admits a convolution kernel $k_T \in \Dist'(G)$, such that
\[
T f(x) = f*k_T(x) = \int_G f(xy^{-1}) \,k_T(y) \di\mu(y);
\]
in this case the convolution kernel $k_T$ is related to the integral kernel $K_T^\mu$ by
\begin{equation}\label{eq:conv_int_kernel}
K_T^\mu(x,y) = k_T(y^{-1} x) \, m(y).
\end{equation}

Let $\Delta = -\sum_{j=1}^\nu X_j^2$ be the ``sum of squares'' sub-Laplacian associated with the H\"ormander system $X_1,\dots,X_\nu$. The sub-Laplacian $\Delta$ is a left-invariant, nonnegative, essentially self-adjoint operator on $L^2(\mu)$, and actually
\begin{equation}\label{eq:sub_nabla_adj}
\Delta = \nabla_H^+ \nabla_H,
\end{equation}
where $\nabla_H^+$ is the formal adjoint of the horizontal gradient $\nabla_H$ with respect to $\mu$. In particular, a functional calculus for $\Delta$ is defined via the spectral theorem: for all bounded Borel functions $F : \RR \to \CC$, the operator $F(\Delta)$ is $L^2(\mu)$-bounded and left-invariant, whence
\[
F(\Delta) f = f * k_{F(\Delta)}
\]
for some convolution kernel $k_{F(\Delta)}$, which in general is only a distribution on $G$. However, $k_{F(\Delta)} \in L^2(\mu)$ when $F$ is bounded and compactly supported, and actually
\begin{equation}\label{eq:plancherel}
\| k_{F(\Delta)} \|_{L^2(\mu)} = \| k \|_{L^2(\sigma_\Delta)}
\end{equation}
for some regular positive Borel measure $\sigma_\Delta$ on $\RR$, called the Plancherel measure associated to $\Delta$, whose support is the $L^2(\mu)$-spectrum of $\Delta$ (see, e.g., \cite[Section 3.2]{martini_spectral_2011}). Moreover, due to \eqref{eq:sub_nabla_adj}, the sub-Riemannian distance $\dist$ is the ``control distance'' for the operator $\Delta$, and in particular finite propagation speed holds (see, e.g., \cite{melrose_propagation_1986} or \cite{cowling_subfinsler_2013}):
for all $t \in \RR$,
$\supp k_{\cos(t\sqrt{\Delta})} 
\subseteq B_\dist(e,|t|)$.
As it is well known (cf. \cite{CGT} or \cite[Lemma 2.1]{cowling_spectral_2001}), via the Fourier inversion formula this implies that
\begin{equation}\label{eq:fps}
\supp k_{F(\sqrt{\Delta})} \subseteq B_\dist(e,r)
\end{equation}
for all $r \in \Rpos$ and all even functions $F : \RR \to \CC$ with $\supp \hat F \subseteq [-r,r]$. Here $\hat F$ denotes the Fourier transform of $F$, defined by  $\hat F(\xi) = \int_\RR F(x) \,e^{-i x \xi} \di x$.

Let $X$ be a nonzero left-invariant vector field, and define the ``sub-Laplacian with drift'' $\Delta_X = \Delta - X$. By \cite[Proposition 3.1]{HMM}, we know that $\Delta_X$ is formally self-adjoint with respect to a positive measure $\tilde\mu$ on $G$ if and only if there exists a positive character $\chi$ of $G$ such that
\[
\nabla_H \chi|_e = X|_e,
\]
which we assume from now on. In this case $\tilde \mu$ is a multiple of the measure $\mu_X = \chi \mu$ and $\Delta_X$ is essentially self-adjoint on $L^2(\mu_X)$.
Moreover
\[
\chi^{1/2} \Delta_X (\chi^{-1/2} f) = (\Delta + b_X^2) f
\]
for all $f \in C^\infty_c(G)$, where $b_X = |X|/2$ \cite[eq.\ (3.6)]{HMM}.  Since $L^2(\mu) \ni f \mapsto \chi^{-1/2} f \in L^2(\mu_X)$ is an isometric isomorphism, this implies a relation between the functional calculi of $\Delta_X$ on $L^2(\mu_X)$ and of $\Delta$ on $L^2(\mu)$: for all bounded Borel functions $F : \RR \to \CC$ and $f \in L^2(\mu)$,
\[
\chi^{1/2} F(\Delta_X) (\chi^{-1/2} f) = F(\Delta + b_X^2) f.
\]
This shows that the $L^2(\mu_X)$-spectrum of $\Delta_X$ is contained in $[b_X^2,\infty)$ and
\begin{equation}\label{eq:rel_cv_ops}
k_{F(\Delta_X)} = \chi^{-1/2} k_{F(\Delta + b_X^2)}
\end{equation}
for all bounded Borel functions $F : \RR \to \CC$.

Information on the growth of the character $\chi$ 
in terms of the sub-Riemannian distance is given by the following result.

\begin{lem}\label{lem:char_estimate}
For all $x \in G$,
\[
\chi(x) \leq e^{|X| \, |x|_\dist}.
\]
Moreover, if $G$ has polynomial growth of degree $d_\infty \geq 1$ and \eqref{eq:poly_sphere_bound} holds for some $\delta \in (0,1]$, then
\[
\int_{B_{\dist}(e,r)} \chi(x) \di\mu(x) \lesssim r^{d_\infty-\delta} e^{|X| r}
\]
for all $r \geq 1$.
\end{lem}
\begin{proof}
The pointwise estimate on $\chi$ is proved in \cite[Proposition 5.7(ii)]{HMM}. As for the integral estimate, note that, by \eqref{eq:poly_infty_asymp} and \eqref{eq:poly_sphere_bound},
\[\begin{split}
\int_{B_{\dist}(e,r)} \chi(x) \di\mu(x) &\leq \sum_{j=1}^{\lceil r \rceil} \mu(B_\dist(e,j) \setminus B_\dist(e,j-1)) \, e^{|X| j} \\
&\lesssim \sum_{j=1}^{\lceil r \rceil} j^{d_\infty-\delta} e^{|X| j} \lesssim r^{d_\infty-\delta} e^{|X| r}
\end{split}\]
for all $r \geq 1$ (in the last inequality the fact that $|X| > 0$ is used).
\end{proof}

Moreover, from Proposition \ref{prp:basic_sub_riemannian} we immediately deduce some important properties of the measured metric space $(G,\dist,\mu_X)$.

Define $\Balls=\{B_\dist(c, r) \tc c \in G, \, r>0\}$ and $\Balls_b=\{B_\dist(c,r) \tc c \in G, \, 0 < r \leq b\}$ for all $b \in \Rpos$. For a ball $B \in \Balls$, we denote by $2B$ the ball with the same centre and twice the radius.

\begin{lem}\label{lem:geom}
The space $(G,\dist, \mu_{X} )$ satisfies the following properties.
\begin{enumerate}[label=(\roman*)]
\item\label{en:geom_ldp} {\emph{Local doubling property}}: for every $b>0$ there exists $D_b \in \Rpos$ such that
\[
\mu_{X}(2B)\leq D_b\,\mu_{X}(B)\qquad \forall B\in\Balls_b.
\]
\item\label{en:geom_midp} 
{\emph{Midpoint property}}:
for all $x,y\in G$, there exists $z \in G$ such that 
\[
\dist(x,z) = \dist(z,y) = \dist(x,y)/2.
\]
\end{enumerate}
\end{lem} 
\begin{proof}
\ref{en:geom_ldp}. 
From the behaviour of $V_\dist(r)$, as described in Proposition \ref{prp:basic_sub_riemannian}, it is clear that, for every $b \in \Rpos$, there exists $C_b \in \Rpos$ so that
\begin{equation}\label{eq:locallydoubling}
V_{\dist}(2r)\leq  C_b\, V_{\dist}({r}) \qquad \forall r\in (0,b].
\end{equation}

Note now that, for all $B=B_\dist(c_B,r_B)\in\Balls$,
\[
\begin{aligned}
\mu_{X}(B)&=\int_{B_\dist(c_B,r_B)}\chi(x)  \di\mu (x)\\
&=\int_{B_\dist(e,r_B)}\chi(c_By)\,  m^{-1}(c_B)\di\mu (y)\\
&=\chi(c_B)\,m^{-1}(c_B)\mu_X\big(B_\dist(e,r_B)\big).
\end{aligned}
\]
Hence, for all $b > 0$ and for all $B \in \Balls_b$,
\[
\mu_{X}(B)\geq   \chi(c_B)\,m^{-1}(c_B) V_\dist(r_B)  \,\inf_{B_\dist(e,b)}\chi,
\]
and
\[
\mu_{X}(2B)\leq   \chi(c_B)\,m^{-1}(c_B) V_\dist(2r_B) \,\sup_{B_\dist(e,2b)} \chi.
\]
Thus, if we set $A_b = \sup_{B_\dist(e,2b)} \chi / \inf_{B_\dist(e,b)} \chi$, then, for every $B\in \Balls_b$,
\[
\frac{\mu_{X}(2B)}{\mu_{X}(B)}
 \leq A_b C_b,
\]
by \eqref{eq:locallydoubling}, and part \ref{en:geom_ldp} follows. 

\ref{en:geom_midp}. This is an immediate consequence of the fact that $(G,\dist)$ is a complete locally compact length space (see, e.g., \cite[Lemma 2.4.8 and Theorem 2.5.23]{BBI}).
\end{proof}

The previous lemma shows that one can apply the theory of Hardy spaces of Goldberg type developed in \cite{MV} to the space $(G,\dist,\mu_X)$. For the reader's convenience, we recall here briefly the definition of the atomic Hardy space $\hardy(\mu_{X})$ and its dual $\bmo(\mu_{X})$ and a few related results. We refer the reader to \cite{G} for details on the theory of Goldberg Hardy spaces in the Euclidean setting and to \cite{MV} for the corresponding theory in the context of metric spaces.
 
\begin{defi}
A standard atom at scale $1$ is a function $a\in L^1(\mu_{X})$ supported in a ball $B\in \Balls_1$ such that
\begin{enumerate}[label=(\roman*)]
\item $\|a\|_{L^2(\mu_{X})}\leq \mu_{X}(B)^{-1/2}$;
\item $\int a\di\mu_{X}=0$\,.
\end{enumerate}
A global atom at scale $1$ is a function $a\in L^1(\mu_{X})$ supported in a ball $B\in \Balls_1$ such that $\|a\|_{L^2(\mu_{X})}\leq \mu_{X}(B)^{-1/2}$. 
Standard and global atom at scale $1$ will be referred to as atoms at scale $1$. 
\end{defi}

\begin{defi}
The Hardy space $\hardy(\mu_{X})$ is defined as the space 
\[
\hardy(\mu_{X})=\Big\{f\in L^1(\mu_{X}): f=\sum_{k}c_ka_k, a_k\,\,{\rm{atoms \,\,at\,\, scale\,\,}} 1, c_k\in\mathbb C,\,\sum_k|c_k|<\infty \Big\}\,,
\]
endowed with the usual atomic norm
\[
\|f\|_{\hardy(\mu_{X})}=\inf\Big\{ \sum_k|c_k|: f=\sum_{k}c_ka_k\Big\}\,.
\]
\end{defi}

By \cite[Theorem 2]{MV} the dual of $\hardy(\mu_{X})$ can be identified with a suitably defined BMO space $\bmo(\mu_{X})$.

\begin{defi}
$\bmo(\mu_{X})$ is the space of all equivalence classes of locally integrable functions $g$ modulo constants such that
\[
\begin{aligned}
\|g\|_{\bmo(\mu_{X})}&:=\sup_{B\in\Balls_1}\Big(\frac{1}{\mu_X(B)} \int_B|g-g_B|^2\di\mu_{X}\Big)^{1/2} \\& +  \sup_{x\in G}\Big(\frac{1}{\mu_X(B_\dist(x,1))} \int_{B_\dist(x,1)}|g|^2\di\mu_{X}\Big)^{1/2}  <\infty\,,
\end{aligned}
\]
where $g_B = \mu_X(B)^{-1} \int_B g \di \mu_X$.
\end{defi}

By
 \cite[Theorem 8.2]{CMM1} and \cite[Proposition 4.5]{CMM2} the following criterion for the boundedness of integral operators on $G$ holds.

\begin{prop}\label{prp:hardy_hormander}
If $T$ is a  bounded operator on $L^2(\mu_X)$ and its integral kernel
 $K_T^{\mu_X}$ is a locally integrable function off the diagonal of $G\times G$ such that 
\begin{equation}\label{eq:hormander-cmm}
N^X_1(T) := \sup_{B\in \Balls_1} \sup_{y,z\in B} \int_{(2B)^c}|K_T^{\mu_X}(x,y)-K_T^{\mu_X}(x,z)| \di \mu_X(x) < \infty
\end{equation} 
and
\begin{equation}\label{eq:hormander-mv}
N^X_2(T) :=  \sup_{y\in G} \int_{(B_\dist(y,2))^c}|K_T^{\mu_X}(x,y)| \di \mu_X(x) < \infty ,
\end{equation} 
then $T$ is bounded from $\hardy(\mu_X)$ to $L^1(\mu_X)$, with
\[
\|T\|_{\hardy(\mu_X) \to L^1(\mu_X)} \leq \max \{ N^X_1(T), N^X_2(T) \} + D_2^{1/2}  \|T\|_{L^2(\mu_X) \to L^2(\mu_X)},
\]
where $D_2$ is the local doubling constant of Lemma \ref{lem:geom}\ref{en:geom_ldp}.
\end{prop}

Furthermore, by  \cite[Theorem 5]{MV}, the following interpolation result holds, where $(V,W)_{[\theta]}$ denotes the lower complex interpolation space of parameter $\theta \in (0,1)$ between the Banach spaces $V,W$.

\begin{teo}\label{thm:hardy_interpol}
Let $\theta \in (0,1)$ and set $p_\theta = 2/(2-\theta)$. Then $(\hardy(\mu_X),L^2(\mu_X))_{[\theta]} = L^{p_\theta}(\mu_X)$ and $(\bmo(\mu_X),L^2(\mu_X))_{[\theta]} = L^{p'_\theta}(\mu_X)$. 
\end{teo}

\subsection{Spaces of smooth and holomorphic functions}\label{s:smooth}
In order to state and prove our multiplier theorems for the sub-Laplacian with drift $\Delta_X$, we need to introduce certain spaces of functions on $\RR$ (and domains of $\CC$) that describe the smoothness conditions that we are going to require on the multiplier. Much of the theory discussed in this section has been developed in \cite{anker}, to which we refer for further discussion and details; we remark that \cite{anker} treats, more generally, spaces of functions on $\RR^a$ for an arbitrary dimension $a$, so our discussion here refers to the results in \cite{anker} with $a = 1$.

For $q \in [1,\infty]$ and $\sigma \in \RR$, let $\AH^\sigma_q$ denote the $L^q$ Sobolev (or Bessel potential) space on $\RR$ of fractional order $\sigma$ (see, e.g., \cite[Definition 6.22]{BL}).
Let $\psi$ be a nonnegative function in $C^{\infty}_c(\RR)$, supported in $[1/4,4]$, such that
\begin{equation}\label{eq:sumpsi}
\sum_{j\in\ZZ}\psi(2^{j}\lambda)=1\qquad \forall \lambda \in (0,\infty).
\end{equation}
Similarly as in \cite[Appendix B]{anker}, we define a class of weighted Sobolev spaces on $\RR$ as follows:
for all $\sigma,\tau \in \RR$, $q \in (1,\infty)$ and $r \in [1,\infty]$,
\[
\|F\|_{\AH_{q,r}^{\sigma,\tau}}  = \begin{cases}
\left(\sum_{k \in \NN} \left (2^{k(\tau+1/q)} \| F(2^k \cdot) \psi_{(k)} \|_{\AH^{\sigma}_q} \right)^r \right)^{1/r} &\text{if $r < \infty$,}\\
\sup_{k \in \NN} 2^{k(\tau+1/q)} \| F(2^k \cdot) \psi_{(k)} \|_{\AH^{\sigma}_q}  &\text{if $r=\infty$,} \\
\end{cases}
\]
where 
 $\psi_{(k)} = \sum_{\epsilon=\pm 1}\psi(\epsilon \, \cdot)$ if $k > 0$ and $\psi_{(0)} = \sum_{\epsilon = \pm 1,\, j \in \NN} \psi(\epsilon \, 2^j \cdot)$. It is not difficult to check that the above definition is essentially independent of the choice of the cutoff $\psi$ (that is, different choices of $\psi$ give rise to equivalent $\AH^{\sigma,\tau}_{q,r}$ norms)
 and is equivalent to the one given in \cite[Appendix B]{anker}.
 Moreover, the $\AH^{\sigma,\tau}_{q,r}$ are Banach spaces \cite[Proposition 23(i)]{anker} and the following interpolation result holds \cite[Proposition 23(v)]{anker}, where $(\cdot,\cdot)_{[\theta]}$ and $(\cdot,\cdot)^{[\theta]}$ refer to the lower and upper complex interpolation methods respectively.

\begin{lem}\label{lem:AHinterpol}
For all $\theta \in (0,1)$, $q_0,q_1 \in (1,\infty)$, $\sigma_0,\sigma_1,\tau_0,\tau_1 \in \RR$, $r_0,r_1 \in [1,\infty]$,
\begin{align*}
(\AH_{q_0,r_0}^{\sigma_0,\tau_0},\AH_{q_1,r_1}^{\sigma_1,\tau_1})_{[\theta]} &= \AH_{q_\theta,r_\theta}^{(1-\theta)\sigma_0+\theta\sigma_1,(1-\theta)\tau_0+\theta\tau_1} \qquad\text{if } \min\{r_0,r_1\} < \infty, \\
(\AH_{q_0,\infty}^{\sigma_0,\tau_0},\AH_{q_1,\infty}^{\sigma_1,\tau_1})^{[\theta]} &= \AH_{q_\theta,\infty}^{(1-\theta)\sigma_0+\theta\sigma_1,(1-\theta)\tau_0+\theta\tau_1}, 
\end{align*}
where $1/q_\theta = (1-\theta)/q_0+\theta/q_1$ and $1/r_\theta = (1-\theta)/r_0+\theta/r_1$.
\end{lem}

It is useful to compare the above-defined spaces with other spaces of functions with integer order of differentiability.
For all $q \in [1,\infty)$, $N \in \NN$ and $\tau \in \RR$, we define the space $\AW_q^{N,\tau}$ of locally integrable functions $F : \RR \to \CC$ with locally integrable distributional derivatives up to order $N$ such that
\[
\|F\|_{\AW_q^{N,\tau}} := \max_{k\in\{0,\dots,N\}} \left(\int_\RR  | (1+|\lambda|)^{k+\tau} \, F^{(k)}(\lambda)|^q \di\lambda\right)^{1/q} < \infty.
\]
Similarly, for all $N \in \NN$ and $\tau \in \RR$, we define the space $\AC^{N,\tau}$ of $N$ times continuously differentiable functions $F : \RR \to \CC$ such that
\[
\|F\|_{\AC^{N,\tau}} := \max_{k\in\{0,\dots,N\}} \sup_{\lambda \in \RR} |(1+|\lambda|)^{k+\tau} F^{(k)}(\lambda)|  < \infty.
\]
Note that the condition $\|F\|_{\AC^{N,0}} < \infty$ can be thought of a nonhomogeneous pointwise Mihlin--H\"ormander condition of order $N$ on the function $F : \RR \to \CC$. Similarly, the condition $\| F \|_{\AH^{\sigma,-1/q}_{q,\infty}} < \infty$ can be thought of as a nonohomogeneous $L^q$ Mihlin--H\"ormander condition of order $\sigma$.

\begin{lem}\label{lem:AHemb}
Let $q_0,q_1,q \in (1,\infty)$, $\sigma_0,\sigma_1,\sigma,\tau_0,\tau_1,\tau,s \in \RR$, $r_0,r_1,r \in [1,\infty]$, $N \in \NN$. Then the following continuous embeddings hold.
\begin{enumerate}[label=(\roman*)]
\item\label{en:AHemb_sigma} $\AH_{q,r}^{\sigma_0,\tau} \subseteq \AH_{q,r}^{\sigma_1,\tau}$ if $\sigma_0 \geq \sigma_1$;
\item\label{en:AHemb_rtau} $\AH_{q,r_0}^{\sigma,\tau_0} \subseteq \AH_{q,r_1}^{\sigma,\tau_1}$ if either $\tau_0 > \tau_1$, or both $\tau_0=\tau_1$ and $r_0 \leq r_1$;
\item\label{en:AHemb_q1} $\AH^{\sigma,\tau}_{q_0,r} \subseteq \AH^{\sigma,\tau -(1/q_1-1/q_0)}_{q_1,r}$ if $q_0 \geq q_1$;
\item\label{en:AHemb_q2} $\AH^{\sigma,\tau}_{q_0,r} \subseteq \AH^{\sigma - (1/q_0-1/q_1),\tau + (1/q_0-1/q_1)}_{q_1,r}$ if $q_0 \leq q_1$;
\item\label{en:AHemb_C2} $\AH_{q,r}^{\sigma,\tau} \subseteq \AC^{N,\tau+1/q}$ if $\sigma > N+1/q$;
\item\label{en:AHemb_mult} $\AC^{N,0} \cdot \AH_{q,r}^{\sigma,\tau} \subseteq \AH_{q,r}^{\sigma,\tau}$ if $N > \sigma$;
\item\label{en:AHemb_C1} $\AC^{N,0} \subseteq \AH_{q,\infty}^{\sigma,-1/q}$ if $N > \sigma$;
\item\label{en:AHemb_conv} $\AW_1^{0,s} * \AH_{q,r}^{\sigma,\tau} \subseteq \AH_{q,r}^{\sigma,\tau}$ if $s > |\sigma|+|\tau|$;
\item\label{en:AHemb_W} $\AW_q^{N,\tau}= \AH_{q,q}^{N,\tau}$.
\end{enumerate}
\end{lem}
\begin{proof}
These embeddings are stated, in a slightly different form, in \cite[Proposition 23(iii,vii,x,xii)]{anker}. We note in particular that part \ref{en:AHemb_C1} follows from part \ref{en:AHemb_mult} by observing that constant functions belong to $\AH_{q,\infty}^{\sigma,-1/q}$.
\end{proof}

We now switch to the discussion of spaces of holomorphic functions on domains of $\CC$. As usual, for all open subsets $\Omega \subseteq \CC$, we denote by $H^\infty(\Omega)$ the space of all bounded holomorphic functions on $\Omega$.

For all $W \in (0,\infty)$ and $N \in \NN$ we denote by $\Sigma_W$ the complex strip defined by 
$\Sigma_W=\{x+iy\in\CC: |y|<W\}$ and by $H^{\infty}(\Sigma_W;N)$ the space of even bounded holomorphic functions $F$ on the strip 
$\Sigma_W$, continuous on the closure of $\Sigma_{W}$ with all their derivatives up to order $N$, such that
\[
\|F\|_{H^\infty(\Sigma_W;N)} := \max_{j\in\{0,\dots,N\}}\sup_{z \in \overline{\Sigma_W}} (1+|z|)^{j} |F^{(j)}(z)| < \infty \,.
\]

Similarly as in \cite[Section 3]{anker}, for all $q \in (1,\infty)$, $r \in [1,\infty]$, $\sigma \in (1/q,\infty)$, $\tau \in \RR$, we denote by $\AcH^{\sigma,\tau,W}_{q,r}$ the  space of all even continuous
functions $F : \overline{\Sigma_W} \to \CC$ that are holomorphic in $\Sigma_W$, have at most polynomial growth, and are such that $F(\cdot \pm i W) \in \AH^{\sigma,\tau}_{q,r}$.
In the case $W = 0$, we define $\AcH^{\sigma,\tau,0}_{q,r}$ as the space of even functions belonging to $\AH^{\sigma,\tau}_{q,r}$.
For all $W \in [0,\infty)$, the space $\AcH^{\sigma,\tau,W}_{q,r}$ is a Banach space \cite[Lemma 9]{anker} with the norm
\[
\|F\|_{\AcH^{\sigma,\tau,W}_{q,r}} = \|F(\cdot + i W)\|_{\AH^{\sigma,\tau}_{q,r}} = \|F(\cdot - i W)\|_{\AH^{\sigma,\tau}_{q,r}}
\]
(the latter equality being due to parity), and actually, due to the three-lines theorem
(see \cite[proof of Lemma 9(i)]{anker}),
\begin{equation}\label{eq:AcHthreelines}
\|F\|_{\AcH^{\sigma,\tau,W}_{q,r}} \sim \sup_{t \in [-W,W]} \|F(\cdot + i t)\|_{\AH^{\sigma,\tau}_{q,r}}.
\end{equation}
The spaces $H^\infty(\Sigma_W;N)$ and $\AcH^{\sigma,\tau,W}_{q,r}$ can be thought of the holomorphic counterparts of the spaces $\AC^{N,0}$ and $\AH^{\sigma,\tau}_{q,r}$, and the embeddings of Lemma \ref{lem:AHemb} imply corresponding embeddings of spaces of holomorphic functions.
Moreover the following interpolation result holds \cite[Lemma 10]{anker}.

\begin{lem}\label{lem:AcHinterpol}
For all $\theta \in (0,1)$, $q_0,q_1 \in (1,\infty)$, $\sigma_0 \in (1/q_0,\infty)$, $\sigma_1 \in (1/q_1,\infty)$, $\tau_0,\tau_1 \in \RR$, $W_0,W_1 \in [0,\infty)$,
\begin{equation}\label{eq:AcHinterpolation}
(\AcH^{\sigma_0,\tau_0,W_0}_{q_0,q_0},\AcH^{\sigma_1,\tau_1,W_1}_{q_1,q_1})_{[\theta]} = \AcH^{(1-\theta)\sigma_0+\theta\sigma_1,(1-\theta)\tau_0 +\theta\tau_1,(1-\theta)W_0+\theta W_1}_{q_\theta,q_\theta},
\end{equation}
where $1/q_\theta = (1-\theta)/q_0+\theta/q_1$.
\end{lem}

Recall that, for every integrable function $F : \RR \to \CC$, we denote by $\hat F$ its Fourier transform, given by $\hat F(\xi) = \int_\RR F(x) \,e^{-i\xi x} \di x$. If $F$ is a holomorphic function on a domain containing $\RR$, then $\hat F$ denotes the Fourier transform of $F|_\RR$. As it is well-known, if $\hat F$ is compactly supported, then $F$ uniquely extends to an entire function, that we still denote as $F$.

 Let $\omega : \RR \to \RR$ be an even smooth cutoff function supported in $[-1,1]$, equal to $1$ in $[-1/4,1/4]$ and such that 
\begin{equation}\label{eq:omega}
\sum_{h\in\mathbb Z} \omega(t-h)=1\qquad \forall t\in\mathbb R\,.
\end{equation}
For all $h \in \NN$, $h\geq 2$, define $\omega_h : \RR \to \RR$ by 
\begin{equation}\label{eq:omegah}
\omega_h(t)=\omega(t-h+1)+\omega(t+h-1).
\end{equation}
Note that ${\rm supp}\,\omega_h\subset [h-2,h]\cup [-h,-h+2]$.
The following estimates of Paley--Wiener type should be compared with \cite[Lemma 5.4]{HMM} and \cite[eq.\ (26)]{anker}.

\begin{lem}\label{lem:paleywiener}
Let $q \in (1,\infty)$, $\sigma,b \in [0,\infty)$ and $W \in (0,\infty)$.
\begin{enumerate}[label=(\roman*)]
\item\label{en:paleywiener1} For all $R \in (0,\infty)$, if $F \in \AcH^{\sigma,b-\sigma,W}_{q,q}$, and
$\Rnon \cap \supp \hat F \subseteq [R,\infty)$, then
\[
\|(1+|\cdot|)^b F\|_{L^q(\RR)} \leq C_{\sigma,b,W} \, R^{-\sigma} e^{-WR} \|F\|_{\AcH^{\sigma,b-\sigma,W}_{q,q}}.
\]
\item\label{en:paleywiener2} If $F \in \AcH^{\sigma,b-\sigma,W}_{q,q}$, then, for all $h \in \NN$, $h \geq 3$,
\[
\|(1+|\cdot|)^b F_h\|_{L^q(\RR)} \leq C_{\sigma,b,W} \, h^{-\sigma} e^{-Wh} \|F\|_{\AcH^{\sigma,b-\sigma,W}_{q,q}},
\]
where $F_h$ is defined by $\hat F_h = \omega_h \hat F$.
\end{enumerate}
\end{lem}
\begin{proof}
By interpolation, it is enough to consider the case where $\sigma,b \in \NN$. By an approximation argument, exploiting \cite[Proposition 23(xi)]{anker}, we may further assume that $\supp \hat F$ is compact, so $F$ extends to an entire function.

We need two preliminary inequalities.
The first is that, for all
$G : \RR \to \CC$ with $\supp \hat G \Subset \RR \setminus (-R,R)$ and all $k \in \NN$,
\begin{equation}\label{eq:der_sp_loc}
\| G \|_{L^q(\RR)} \leq C_{q,k} \, R^{-k} \| G^{(k)} \|_{L^q(\RR)} ;
\end{equation}
this is an easy consequence of the Mihlin--H\"ormander Fourier multiplier theorem on $\RR$ (cf.\ \cite[Lemma 6.2.1]{BL}). The second is that, for all $G : \RR \to \CC$ with $\supp \hat G \Subset \RR \setminus (-R,R)$ and $W \in (0,\infty)$,
\begin{equation}\label{eq:paleywiener}
\| G \|_{L^q(\RR)} \leq C_q \, e^{-W R} (\| G(\cdot + iW) \|_{L^q(\RR)} + \| G(\cdot - iW) \|_{L^q(\RR)});
\end{equation}
indeed, by decomposing $G = G_- + G_+$ so that $\supp \hat G_- \subseteq (-\infty,-R]$ and $\supp \hat G_+ \subseteq [R,\infty)$, one can easily write
\begin{align*}
G_- &= e^{-WR} G_-(\cdot + iW) * (e^{iR \cdot} P_W), \\
G_+ &= e^{-WR} G_+(\cdot - iW) * (e^{-iR \cdot} P_W),
\end{align*}
where 
$P_W$ 
is the Poisson kernel (i.e., $\hat P_W(\xi) = e^{-W|\xi|}$), and moreover
\[
G_- = G * (R \alpha(-R\cdot)), \qquad G_+ = G * (R \alpha(R\cdot)),
\]
where
$\alpha$ is chosen so that $\hat \alpha \in C^\infty(\RR)$, $\supp \hat \alpha \subseteq (0,\infty)$ and $\hat \alpha|_{[1,\infty)} \equiv 1$; so, by Young's inequality,
\[\begin{split}
\|G_\pm\|_{L^q(\RR)} 
&\leq e^{-WR} \|G_\pm(\cdot \mp iW)\|_{L^q(\RR)} \|P_W\|_{L^1(\RR)} \\
&\leq e^{-WR} \|G(\cdot \mp iW)\|_q \|P_W\|_{L^q(\RR)} \|\alpha\|_{Cv^q(\RR)};
\end{split}\]
since $\|P_W\|_{L^1(\RR)} = 1$, and moreover $\|\alpha\|_{Cv^q(\RR)} < \infty$ by the Mihlin--H\"ormander theorem, we conclude that
\[\begin{split}
\|G\|_{L^q(\RR)} 
&\leq \|G_-\|_{L^q(\RR)} + \|G_+\|_{L^q(\RR)} \\
&\leq C_q \, e^{-WR} (\| G(\cdot + iW) \|_{L^q(\RR)} + \| G(\cdot - iW) \|_{L^q(\RR)})
\end{split}\]
(cf.\ \cite[proof of Theorem 4.10]{CDMY}).

Note now that, if $F : \RR \to \CC$ is even and $\Rnon \cap \supp \hat F \Subset [R,\infty)$, then, for all $j,k \in \NN$, the functions $F_{j,k}(x) = \partial_x^k ( x^j F(x) )$ are either even or odd and satisfy the same Fourier support condition $\Rnon \cap \supp \hat F_{j,k} \Subset [R,\infty)$; hence, by \eqref{eq:der_sp_loc} and \eqref{eq:paleywiener},
\[\begin{split}
\|(1+|\cdot|)^b F\|_{L^q(\RR)} &\sim \max_{j \in \{0,\dots,b\}} \| F_{j,0} \|_{L^q(\RR)} \\
&\lesssim R^{-\sigma} \max_{j \in \{0,\dots,b\}} \| F_{j,\sigma} \|_{L^q(\RR)} \\
&\lesssim R^{-\sigma} e^{-WR} \max_{j \in \{0,\dots,b\}} \| F_{j,\sigma}(\cdot+iW) \|_{L^q(\RR)} .
\end{split}\]
Therefore, by Leibniz' rule,
\[\begin{split}
\|(1+|\cdot|)^b F\|_{L^q(\RR)} 
&\lesssim R^{-\sigma} e^{-WR} \max_{j \in \{0,\dots,b\}} \| D^\sigma [(\cdot+iW)^j F(\cdot+iW)] \|_{L^q(\RR)} \\
&\leq R^{-\sigma} e^{-WR} \max_{\substack{ j \in \{0,\dots,b\} \\ k \in \{0,\dots,\sigma\} \\ \sigma \leq j+k}}  \| (\cdot+iW)^{j-\sigma+k} F^{(k)}(\cdot+iW) \|_{L^q(\RR)}\\
&\lesssim R^{-\sigma} e^{-WR} \max_{k \in \{0,\dots,\sigma\}}  \| (1+|\cdot|)^{b-\sigma+k} F^{(k)}(\cdot+iW) \|_{L^q(\RR)} \\
&\sim R^{-\sigma} e^{-WR} \| F(\cdot+iW) \|_{\AH^{\sigma,b-\sigma}_{q,q}},
\end{split}\]
which proves part \ref{en:paleywiener1} (in the last step Lemma \ref{lem:AHemb}\ref{en:AHemb_W} was used).

As for part \ref{en:paleywiener2}, it is sufficient to apply part \ref{en:paleywiener1} to $F_h$ in place of $F$ and $R = h-2$, and observe that, by Lemma \ref{lem:AHemb}\ref{en:AHemb_conv}, $\|F_h(\cdot +iW)\|_{\AH^{\sigma,b-\sigma}_{q,q}} \lesssim \|F(\cdot +iW)\|_{\AH^{\sigma,b-\sigma}_{q,q}}$ with an implicit constant independent of $h$.
\end{proof}

\section{A conditional multiplier theorem for sub-Laplacians with drift}\label{s:main}

Let $\Delta_X = \Delta - X$ be a sub-Laplacian with drift as described in Section \ref{s:preliminaries}, with $X|_e = \nabla_H \chi|_e$ for some nontrivial positive character $\chi$ on $G$.

Here we prove a multiplier theorem for the sub-Laplacian with drift $\Delta_X$.  The result is conditional, in the sense that it is proved under certain assumptions, that are essentially related to the existence of a differentiable $L^p$-functional calculus
for the sub-Laplacian without drift $\Delta$. In Section \ref{s:applications} we will show how these assumptions can be verified in a number of cases.

For notational simplicity, it is convenient to express our result in terms of the operators
\[
\DD = \sqrt{\Delta}, \qquad \DD_X = \sqrt{\Delta_X - b_X^2}.
\]
Note that the passage from $\Delta_X$ to $\DD_X$ corresponds to the change of variable \eqref{eq:change_var}.

\begin{assm}
There exist $\lebexp \in [2,\infty)$ and $\sigma \in (1/\lebexp,\infty)$, such that the following inequalities hold for all even functions $F : \RR \to \CC$ with $\supp \hat F \subseteq [-2,2]$.
\begin{enumerate}[label=(\Alph*)]
\item\label{en:ass_MH_diff} $\sup_{y \in B_\dist(e,1)} \int_{|x|_\dist \geq 2|y|_\dist} |k_{F(\DD)}(xy)-k_{F(\DD)}(x)| \di \mu (x) \leq C \| F \|_{\AH^{\sigma,-1/\lebexp}_{\lebexp,\infty}}$.
\item\label{en:ass_MH_annulus} $\sup_{0<r\leq 1} r \int_{|x|_\dist \geq r} |k_{F(\DD)}(x)| \di\mu(x) \leq C \| F \|_{\AH^{\sigma,-1/\lebexp}_{\lebexp,\infty}}$.
\end{enumerate} 
In addition, there exist $\thresh,\plnexp \in [0,\infty)$ and $W \in (0,\infty)$ such that, for all $h \in \NN \cap [3,\infty)$ and all even functions $F : \RR \to \CC$, if $\Rnon \cap \supp \hat F \subseteq [h-2,h]$, then
\begin{enumerate}[label=(\Alph*),resume]
\item\label{en:ass_fps_plancherel}
$\|\chi^{1/2} k_{F(\DD)}\|_{L^1(\mu)} \leq C e^{Wh} h^\thresh \| (1+|\cdot|)^\plnexp \, F \|_{L^\lebexp(\RR)}$.
\end{enumerate}
\end{assm}

Under the above assumptions, we are able to prove the following spectral multiplier theorem for $\Delta_X$. 

\begin{teo}\label{teo:main}
Suppose that Assumptions \ref{en:ass_MH_diff}, \ref{en:ass_MH_annulus}, \ref{en:ass_fps_plancherel} hold for some $\lebexp \in [2,\infty)$, $\sigma \in (1/\lebexp,\infty)$, $\thresh,\plnexp \in [0,\infty)$, $W \in (0,\infty)$.
Let $p \in [1,\infty] \setminus \{2\}$, and let $q \in [\lebexp,\infty)$ be defined by $1/q = |2/p-1|/\lebexp$. Suppose that $M \in \AcH_{q,\infty}^{s,-1/q,|2/p-1|W}$ for some $s \in \RR$ satisfying
\begin{equation}\label{eq:smoothness_order}
s > 2 \, |1/p-1/2| \, \max \{ \sigma, \thresh+1, \plnexp+1/\lebexp\}.
\end{equation}
\begin{enumerate}[label=(\roman*)]
\item\label{en:main_hardy} If $p=1$, then $M(\DD_X)$ extends to a bounded operator from $\hardy(\mu_{X})$ to $L^1(\mu_{X})$.

\item\label{en:main_bmo} If $p=\infty$, then $M(\DD_X)$ extends to a bounded operator from $L^\infty(\mu_{X})$ to $\bmo(\mu_{X})$.

\item\label{en:main_lp} If $p \in (1,\infty)$, then $M(\DD_X)$ extends to a bounded operator on $L^p(\mu_{X})$.
\end{enumerate}
\end{teo}

The above result is stated in terms of an $L^q$ Sobolev condition of fractional order on the multiplier, where $q$ is the exponent obtained by interpolation between $\lebexp$ and $\infty$, corresponding to the choice of $p$ between $1$ and $2$ (or $\infty$ and $2$); in particular, for the endpoint results $p=1,\infty$ the condition has the same  $L^\lebexp$ type that appears in the Assumptions.
For the sake of clarity, we also state the result in a simplified form, involving a pointwise condition of integer order on the multiplier.

\begin{coro}\label{cor:main}
Suppose that Assumptions \ref{en:ass_MH_diff}, \ref{en:ass_MH_annulus}, \ref{en:ass_fps_plancherel} hold for some $\lebexp \in [2,\infty)$, $\sigma \in (1/\lebexp,\infty)$, $\thresh,\plnexp \in [0,\infty)$, $W \in (0,\infty)$.
Let $p \in [1,\infty] \setminus \{2\}$, and set $W_p = 2|1/p-1/2|W$. Suppose that $M \in H^\infty(\Sigma_{W_p};N)$ for some $N \in \NN$ satisfying
\[
N > 2 \, |1/p-1/2| \, \max \{ \sigma, \thresh+1, \plnexp+1/\lebexp\}.
\]
Then the conclusions \ref{en:main_hardy}, \ref{en:main_bmo}, \ref{en:main_lp} of Theorem \ref{teo:main} hold.
\end{coro}
\begin{proof}
If we choose $s \in \RR$ so that
\[
N > s > 2 \, |1/p-1/2| \, \max \{ \sigma, \thresh+1, \plnexp+1/\lebexp\},
\]
then the result follows immediately from Theorem \ref{teo:main}, together with the fact that
\[
\| M \|_{\AcH^{s,-1/q,W_p}_{q,\infty}} \lesssim \| M \|_{H^\infty(\Sigma_{W_p};N)}
\]
(see Lemma \ref{lem:AHemb}\ref{en:AHemb_C1}).
\end{proof}

Before entering the proof of Theorem \ref{teo:main}, some remarks on the above results and assumptions are in order.

\begin{rem}
The above results are stated in terms of the functional calculus of $\DD_X$. 
In order to relate them more directly to the functional calculus of $\Delta_X$, the change of variable \eqref{eq:change_var} can be used. Indeed, if $M$ is a bounded holomorphic function on a domain containing $[b_X^2,\infty)$, then \eqref{eq:change_var}
defines an even bounded holomorphic function $M_X$ whose domain contains $\RR$ and
\[
M_X(\DD_X) = M(\Delta_X).
\]
For every $p\in [1,\infty]\setminus\{2\}$, let $P_{X,p}$ be the parabolic region defined by \eqref{eq:parabola}.
Let the number $W_{X,p}$ be defined by
\begin{equation}\label{eq:strip_width}
W_{X,p} 
= |1/p-1/2| \, |X| .
\end{equation}
Then, for all $M \in H^\infty(P_{X,p})$, the function $M_X$ defined by 
\eqref{eq:change_var}
lies in $H^\infty(\Sigma_{W_{X,p}})$. 
Hence Theorem \ref{teo:main} and Corollary \ref{cor:main}, applied with $W = b_X$, show that functions $M \in H^\infty(P_{X,p})$ satisfying suitable differential conditions are $L^p$ spectral multipliers of $\Delta_X$ when $1 < p < \infty$, with a corresponding Hardy or BMO endpoint result when $p \in \{1,\infty\}$.
\end{rem}

\begin{rem}\label{rem:CZ}
Assumption \ref{en:ass_MH_diff} can be thought of as a version of the H\"ormander condition for Calder\'on--Zygmund singular integral operators, applied to the operator $F(\DD)$. A stronger version of the assumption is the following gradient estimate, resembling Assumption \ref{en:ass_MH_annulus}:
\begin{enumerate}[label=(A$'$)]
\item
$\sup_{0 < r \leq 1} r \int_{|x|_\dist \geq r} |\nabla_H k_{F(\DD)}(x)| \di\mu(x) \leq C \| F \|_{\AH^{\sigma,-1/\lebexp}_{\lebexp,\infty}}$.
\end{enumerate}
Proving that $F(\DD)$ is a Calder\'on--Zygmund operator whenever $F$ satisfies a suitable scale-invariant smoothness condition is a usual strategy in the proof of a multiplier theorem of Mihlin--H\"ormander type for $\Delta$. It is to be noticed, however, that Assumption \ref{en:ass_MH_diff} is a somewhat weaker requirement, in that
the norm of $F$ in the right-hand side corresponds to a nonhomogeneous Mihlin--H\"ormander condition on the multiplier. Moreover, due to the support condition $\supp \hat F \subseteq [-2,2]$ and finite propagation speed, the integrals in Assumptions  \ref{en:ass_MH_diff} and \ref{en:ass_MH_annulus} are actually restricted to a fixed bounded neighbourhood of the identity. Note also that, through a decomposition into dyadic annuli, Assumption \ref{en:ass_MH_annulus} follows from the stronger estimate
\begin{enumerate}[label=(B$'$)]
\item\label{en:ass_MH_annulus2} $\sup_{0<r\leq 1} \int_{r \leq |x|_\dist < 2r} |k_{F(\DD)}(x)| \di\mu(x) \leq C \| F \|_{\AH^{\sigma,-1/\lebexp}_{\lebexp,\infty}}$
\end{enumerate}
and uniform control of the integral over dyadic annuli is another common estimate for Calder\'on--Zygmund kernels (here it is just required for small dyadic annuli).
\end{rem}

\begin{rem}\label{rem:plancherel}
In view of Lemma \ref{lem:char_estimate} and finite propagation speed \eqref{eq:fps}, Assumption \ref{en:ass_fps_plancherel} with $W = |X|/2$ follows from
\begin{enumerate}[label=(C$'$)]
\item\label{en:ass_fps_plancherel2}
$\|k_{F(\DD)}\|_{L^1(\mu)} \leq C h^\thresh \| (1+|\cdot|)^\plnexp \, F \|_{L^\lebexp(\RR)}$.
\end{enumerate}
By replacing \ref{en:ass_fps_plancherel} with \ref{en:ass_fps_plancherel2}, one would obtain a set of assumptions that depend only on the operator $\Delta$ and not on the drift $X$. Moreover, in the case $G$ has polynomial growth of degree $d_\infty$, Assumption \ref{en:ass_fps_plancherel2} with $\thresh=d_\infty$ follows by finite propagation speed \eqref{eq:fps} and the Cauchy--Schwarz inequality from the following estimate, involving the Plancherel measure $\sigma_\Delta$ defined in \eqref{eq:plancherel}:
\begin{enumerate}[label=(C$''$)]
\item\label{en:ass_fps_plancherel3}
$\|F(\sqrt{\cdot})\|_{L^2(\sigma_\Delta)} \leq C \| (1+|\cdot|)^\plnexp \, F \|_{L^\lebexp(\RR)}$.
\end{enumerate}
As we will see in Section \ref{s:applications}, such a relation between Assumption \ref{en:ass_fps_plancherel} and Plancherel-type estimates holds also in some groups of exponential growth (where the role of the dimension at infinity $d_\infty$ in determining the value of $\thresh$ may be played by some ``pseudo-dimension'').
On the other hand, obtaining Assumption \ref{en:ass_fps_plancherel} via Lemma \ref{lem:char_estimate} and the Cauchy--Schwarz inequality from \ref{en:ass_fps_plancherel2} or \ref{en:ass_fps_plancherel3} need not yield the lowest possible values of $\thresh$ and $\plnexp$ (which in turn determine the order of the smoothness condition \eqref{eq:smoothness_order} in the multiplier theorem), whence our preference for Assumption \ref{en:ass_fps_plancherel} over the alternative formulations.
\end{rem}

\begin{rem}\label{rem:linfty_lq}
In contexts such as groups of polynomial growth, one may encounter Plancherel-type estimates of the form \ref{en:ass_fps_plancherel3}, where the $L^\lebexp$ norm in the right-hand side is replaced by an $L^\infty$ norm, that is,
\begin{enumerate}[label=(C$'''$)]
\item\label{en:ass_fps_plancherel4}
$\|k_{F(\sqrt{\cdot})}\|_{L^2(\sigma_\Delta)} \leq C \| (1+|\cdot|)^\plnexp \, F \|_{L^\infty(\RR)}$
\end{enumerate}
(see, e.g., the discussion in \cite[Section 3]{DOS}). Under the support condition $\Rnon \cap \supp \hat F \subseteq [h-2,h]$, however, \ref{en:ass_fps_plancherel4} implies \ref{en:ass_fps_plancherel3}; indeed, if the even function $\alpha \in C^\infty_c(\RR)$ is chosen so that $\alpha|_{[-1,1]} \equiv 1$, and $\alpha_h(t)= \alpha(t-h+1) + \alpha(t+h-1)$, then
\[\begin{split}
\| (1+|\cdot|)^\plnexp \, F \|_{L^\infty(\RR)} &\leq \| (1+|\cdot|)^\plnexp \, \check\alpha_h \|_{L^{\lebexp'}(\RR)} \| (1+|\cdot|)^\plnexp \, F \|_{L^\lebexp(\RR)} \\
&\lesssim \|\alpha\|_{\AH^\plnexp_{\lebexp'}} \| (1+|\cdot|)^\plnexp \, F \|_{L^\lebexp(\RR)}
\end{split}\]
by Young's inequality and translation-invariance of $\AH^\plnexp_{\lebexp'}$, where $\check\alpha_h$ denotes the inverse Fourier transform of $\alpha_h$. Conversely, by the Cauchy--Schwarz inequality,
\[
\| (1+|\cdot|)^\plnexp \, F \|_{L^\lebexp(\RR)} \leq C_{\plnexp,\plnexp'} \| (1+|\cdot|)^{\plnexp'} \, F \|_{L^\infty(\RR)}
\]
whenever $\plnexp' > \plnexp+1/\lebexp$.
In view of this, stating Assumption \ref{en:ass_fps_plancherel} in terms of an $L^\lebexp$ norm on the right-hand side, instead of an $L^\infty$-norm, does not seem to cause a loss of generality: namely, instead of working with $\lebexp=\infty$ directly, it is enough to consider $\lebexp < \infty$ arbitrarily large. Similar considerations allow one to treat the case where the $L^\lebexp$ Sobolev norms in the right-hand side of Assumptions \ref{en:ass_MH_diff} and \ref{en:ass_MH_annulus} are replaced with suitable $L^\infty$ Sobolev norms (cf.\ \eqref{eq:Linftycomp} below).
\end{rem}

The rest of the section is devoted to the proof of Theorem \ref{teo:main}. Note that, by duality, it is enough to consider the case $p < 2$. The first step is proving the endpoint $p=1$, which is contained in Proposition \ref{prp:hardy_boundedness}\ref{en:bothparts} below.

Let $\omega$ and $\omega_h$ be defined as in \eqref{eq:omega}-\eqref{eq:omegah} and define $\eta : \RR \to \RR$ by $\hat\eta=\omega+\omega_2$.

\begin{prop}\label{prp:hardy_boundedness}
Suppose that Assumptions \ref{en:ass_MH_diff}, \ref{en:ass_MH_annulus}, \ref{en:ass_fps_plancherel} hold for some $\lebexp \in [2,\infty)$, $\sigma \in (1/\lebexp,\infty)$, $\thresh,\plnexp \in [0,\infty)$, $W \in (0,\infty)$.
Let $M : \RR \to \CC$ be an even bounded Borel function and decompose it as $M = M_\ell + M_g$, where $M_\ell = \eta * M$.
\begin{enumerate}[label=(\roman*)]
\item\label{en:localpart}
If $M \in \AH^{\sigma,-1/\lebexp}_{\lebexp,\infty}$, 
 then $M_\ell(\DD_X)$ extends to a bounded operator from $\hardy(\mu_{X})$ to $L^1(\mu_{X})$ and
\[
\|M_\ell(\DD_X)\|_{\hardy(\mu_{X}) \to L^1(\mu_{X})} \leq C \| M \|_{\AH^{\sigma,-1/\lebexp}_{\lebexp,\infty}}.
\]
\item\label{en:globalpart}
If $M \in \AcH^{s,\plnexp-s,W}_{\lebexp,\lebexp}$ for some $s > \thresh+1$, 
 then $M_g(\DD_X)$ extends to a bounded operator on $L^1(\mu_{X})$ and
\[
\|M_g(\DD_X)\|_{L^1(\mu_{X}) \to L^1(\mu_{X})} \leq C_{s} \| M \|_{\AcH^{s,\plnexp-s,W}_{\lebexp,\lebexp}}.
\]
\item\label{en:bothparts} If $M \in \AcH^{s,-1/\lebexp,W}_{\lebexp,\infty}$ for some $s$ satisfying
\[
s \geq \sigma, \qquad s > \thresh+1, \qquad s > \plnexp+1/\lebexp,
\]
then $M(\DD_X)$ extends to a bounded operator from $\hardy(\mu_{X})$ to $L^1(\mu_{X})$ and
\[
\|M(\DD_X)\|_{\hardy(\mu_{X}) \to L^1(\mu_{X})} \leq C_s \| M \|_{\AcH^{s,-1/\lebexp,W}_{\lebexp,\infty}}.
\]
\end{enumerate}
\end{prop}
\begin{proof}
By writing $M = M_R + i M_I$, where $M_R(z) = (M(z) + \overline{M(\bar z)})/2$, and noting that 
$M_R$ and $M_I$ belong to any of the spaces $\AH^{\sigma,-1/\lebexp}_{\lebexp,\infty},\AcH^{s,\plnexp-s,W}_{\lebexp,\lebexp},\AcH^{s,-1/\lebexp,W}_{\lebexp,\infty}$ whenever $M$ does, we see that we may assume that $\overline{M(z)} = M(\bar z)$, so in particular the operators
$M(\DD_X),M_\ell(\DD_X),M_g(\DD_X)$ 
are self-adjoint on $L^2(\mu_X)$. 

\bigskip

\ref{en:localpart}.
We are going to show that $M_\ell(\DD_X)$ satisfies the conditions of Proposition \ref{prp:hardy_hormander}. Note first that, since $M \in \AH^{\sigma,-1/\lebexp}_{\lebexp,\infty}$ and $\sigma > 1/\lebexp$,
\begin{equation}\label{eq:local_embeddings}
\|M_\ell\|_\infty \lesssim \| M_\ell \|_{\AH^{\sigma,-1/\lebexp}_{\lebexp,\infty}} \lesssim \| M \|_{\AH^{\sigma,-1/\lebexp}_{\lebexp,\infty}} < \infty,
\end{equation}
by parts \ref{en:AHemb_C2} and \ref{en:AHemb_conv} of Lemma \ref{lem:AHemb}. In particular $M_\ell(\DD_X)$ is bounded on $L^2(\mu_X)$ and
\[
\|M_\ell(\DD_X)\|_{L^2(\mu_X) \to L^2(\mu_X)} \lesssim \| M \|_{\AH^{\sigma,-1/\lebexp}_{\lebexp,\infty}}.
\]

Set $\ell = k_{M_\ell(\DD)}$ and $\ell_X = k_{M_\ell(\DD_X)}$, and 
note that
\begin{equation}\label{eq:rel_conv_kernels}
\ell_X = \chi^{-1/2} \ell
\end{equation}
by \eqref{eq:rel_cv_ops}.
Since $\hat M_\ell = \hat\eta \hat{M}$ and $\hat\eta$ is supported in $[-2,2]$, by finite propagation speed \eqref{eq:fps} we conclude that $\ell$ and $\ell_X$ are supported in $B_\dist(e,2)$. 

By \eqref{eq:int_kernel_meas} and \eqref{eq:conv_int_kernel} and self-adjointness of $M_\ell(\DD_X)$, it is not difficult to check 
that the integral kernel $L_X = K_{M_\ell(\DD_X)}^{\mu_X}$ with respect to $\mu_X$ of the operator $M_\ell(\DD_X)$
satisfies
\begin{equation}\label{eq:rel_int_kernels}
\overline{L_X(x,y)} =
L_X (y,x)=\ell_X(x^{-1}y) \, m(x) \, \chi^{-1}(x).
\end{equation}
Moreover, since $\ell$ and $\ell_X$ are supported in $B_\dist(e,2)$, the integral 
kernel
$L_X$ 
is
 supported in
$\{(x,y)\in G\times G :\, \dist(x,y)\leq 2\}$. Thus 
\[
 \sup_{y\in G} \int_{(B_\dist(y,2))^c}|L_X(x,y)| \di\mu_X(x) =0 ,
\]
and therefore the condition \eqref{eq:hormander-mv} is trivially satisfied, i.e., $N_2^X(M_\ell(\DD_X)) = 0$.

Let $B=B_\dist (c_B,r_B)$, $r_B\leq 1$. Note that, by \eqref{eq:rel_int_kernels},
\[\begin{aligned}
&\sup_{y,z \in B} \int_{(2B)^c}|L_{X}(x,y)-L_{X}(x,z)| \di\mu_X(x) \\
&\leq  2 \sup_{y \in B} \int_{(2B)^c}|L_{X}(x,y)-L_{X}(x,c_B)| \di\mu_X(x) \\
&= 2 \sup_{y \in B} \int_{(2B)^c}|\ell_X(x^{-1}y)-\ell_X(x^{-1} c_B)|m(x)\di \mu(x) \\
&= 2 \sup_{y \in B_\dist(e,r_B)} \int_{(B_\dist(e,2r_B))^c}|\ell_X(xy)-\ell_X(x)| \di \mu(x) 
\end{aligned}\]
and moreover, for all $y \in B_\dist(e,r_B)$, by \eqref{eq:rel_conv_kernels}, the triangle inequality and the support conditions we deduce that
\[\begin{aligned}
&\int_{(B_\dist(e,2r_B))^c}|\ell_X(xy)-\ell_X(x)| \di \mu(x) \\
&\leq \sup_{B_\dist(e,3)} \chi^{-1/2} \, \int_{(B_\dist(e,2r_B))^c}|\ell(xy)-\ell(x)| \di \mu(x) \\
&+ \sup_{B_\dist(e,2)} \chi^{-1/2} \, \int_{(B_\dist(e,2r_B))^c} |\chi^{-1/2}(y)-1| |\ell(x)| \di \mu(x) \\
&\leq \sup_{B_\dist(e,3)} \chi^{-1/2} \, \int_{(B_\dist(e,2r_B))^c}|\ell(xy)-\ell(x)| \di \mu(x) \\
&+ \sup_{B_\dist(e,2)} \chi^{-1/2} \, \sup_{B_\dist(e,1)} |\nabla_H \chi^{-1/2}| \, r_B \int_{(B_\dist(e,2r_B))^c}|\ell(x)| \di \mu(x).
\end{aligned}\]
By \eqref{eq:local_embeddings} and Assumptions \ref{en:ass_MH_diff} and \ref{en:ass_MH_annulus},
we conclude that 
\[
\sup_{y,z \in B} \int_{(2B)^c}|L_{X}(x,y)-L_{X}(x,z)| \di\mu_X(x) \\
 \lesssim \| M \|_{\AH^{\sigma,-1/\lebexp}_{\lebexp,\infty}} < \infty.
\]
 It follows that the condition \eqref{eq:hormander-cmm} is satisfied, with $N_1^X(M_\ell(\DD_X)) \lesssim  \| M \|_{\AH^{\sigma,-1/\lebexp}_{\lebexp,\infty}}$.

By Proposition \ref{prp:hardy_hormander}, we conclude that the operator $M_\ell(\DD_X)$ is bounded from $\hardy(\mu_{X})$ to $L^1(\mu_{X})$, with operator norm bounded by a multiple of $\| M \|_{\AH^{\sigma,-1/\lebexp}_{\lebexp,\infty}}$.
 
\bigskip 

\ref{en:globalpart}
By \eqref{eq:omega} and \eqref{eq:omegah} we can decompose $M_g = \sum_{h \geq 3} P_h$, 
where $\hat{P}_h=\omega_h\,\hat M$. By 
Assumption \ref{en:ass_fps_plancherel} and Lemma \ref{lem:paleywiener},
\[
\|\chi^{1/2} k_{P_h(\DD)}\|_{L^1(\mu)} \leq e^{Wh} \,h^{\thresh}\, \| (1+|\cdot|)^{\plnexp} \, P_h\|_{L^\lebexp(\RR)} \lesssim h^{\thresh-s}\, \|M\|_{\AcH^{s,\plnexp-s,W}_{\lebexp,\lebexp}}.
\]
Since $s>\thresh+1$, by summing over $h$ we deduce that
\[\begin{split}
\|k_{M_g(\DD_X)}\|_{L^1(\mu_X)} &= \|\chi^{1/2} k_{M_g(\DD)}\|_{L^1(\mu)} \\
&\leq  \sum_{h \geq 3} \|\chi^{1/2} k_{P_h(\DD)}\|_{L^1(\mu)} 
\lesssim \|M\|_{\AcH^{s,\plnexp-s,W}_{\lebexp,\lebexp}},
\end{split}\]
where the first equality is due to \eqref{eq:rel_cv_ops}.
Since $M \in \AcH^{s,\plnexp-s,W}_{\lebexp,\lebexp}$, this proves that $M_g(\DD_X)$ is bounded on $L^1(\mu_X)$, with operator norm bounded by a multiple of $\|M\|_{\AcH^{s,\plnexp-s,W}_{\lebexp,\lebexp}}$.

\bigskip

\ref{en:bothparts}.
Since $s \geq \sigma$,
\[
\|M\|_{\AH^{\sigma,-1/\lebexp}_{\lebexp,\infty}} \lesssim \|M\|_{\AH^{s,-1/\lebexp}_{\lebexp,\infty}} \lesssim \|M\|_{\AcH^{s,-1/\lebexp,W}_{\lebexp,\infty}}
\]
by \eqref{eq:AcHthreelines}. Moreover, since 
 $-1/\lebexp > \plnexp-s$,
by Lemma \ref{lem:AHemb}\ref{en:AHemb_rtau}
\[
\|M\|_{\AcH^{s,\plnexp-s,W}_{\lebexp,\lebexp}} \lesssim \|M\|_{\AcH^{s,-1/\lebexp,W}_{\lebexp,\infty}}.
\]
Since $M = M_\ell + M_g$, part \ref{en:bothparts} follows by combining parts \ref{en:localpart} and \ref{en:globalpart}, and observing that boundedness on $L^1(\mu_X)$ implies boundedness from $\hardy(\mu_X)$ to $L^1(\mu_X)$.
\end{proof}

Via interpolation we can now derive Theorem \ref{teo:main} in the remaining case $p \in (1,2)$; the result is contained in Proposition \ref{prp:lp_boundedness}\ref{en:bothparts_int}.

\begin{prop}\label{prp:lp_boundedness}
Suppose that Assumptions \ref{en:ass_MH_diff}, \ref{en:ass_MH_annulus}, \ref{en:ass_fps_plancherel} hold for some $\lebexp \in [2,\infty)$, $\sigma \in (1/\lebexp,\infty)$, $\thresh,\plnexp \in [0,\infty)$, $W \in (0,\infty)$.
Let $q \in (\lebexp,\infty)$.
Let $M : \RR \to \CC$ be an even bounded Borel function and decompose it as $M = M_\ell + M_g$, where $M_\ell = \eta * M$.
\begin{enumerate}[label=(\roman*)]
\item\label{en:localpart_int}
If $M \in \AH^{s,-1/q}_{q,\infty}$ for some $s > \sigma \lebexp/q$, 
 then $M_\ell(\DD_X)$ extends to a bounded operator on $L^p(\mu_{X})$ for all $p$ satisfying $|1/p-1/2| \leq \lebexp/(2q)$ and
\[
\|M_\ell(\DD_X)\|_{L^p(\mu_{X}) \to L^p(\mu_{X})} \leq C_{s,q} \| M \|_{\AH^{s,-1/q}_{q,\infty}}.
\]
\item\label{en:globalpart_int}
If $M \in \AcH^{s,\tau-s,W \lebexp/q}_{q,q}$ for some $s > (\thresh+1) \lebexp/q$ and $\tau > \plnexp \lebexp/q$,
 then $M_g(\DD_X)$ extends to a bounded operator on $L^p(\mu_{X})$ for all $p$ satisfying $|1/p-1/2| \leq \lebexp/(2q)$ and
\[
\|M_g(\DD_X)\|_{L^p(\mu_{X}) \to L^p(\mu_{X})} \leq C_{s,\tau,q} \| M \|_{\AcH^{s,\tau-s,W \lebexp/q}_{q,q}}.
\]
\item\label{en:bothparts_int} If $M \in \AcH^{s,-1/q,W \lebexp/q}_{q,\infty}$ for some $s$ satisfying
\[
s > (\lebexp/q) \max \{\sigma, \thresh+1, \plnexp +1/\lebexp\},
\]
then $M(\DD_X)$ extends to a bounded operator on $L^p(\mu_{X})$ for all $p$ satisfying $|1/p-1/2| \leq \lebexp/(2q)$ and
\[
\|M(\DD_X)\|_{L^p(\mu_{X}) \to L^p(\mu_{X})} \leq C_{s,q} \| M \|_{\AcH^{s,-1/q,W \lebexp/q}_{q,\infty}}.
\]
\end{enumerate}
\end{prop}
\begin{proof}
Note that, by duality and interpolation, it is enough to prove the above boundedness results in the case $1/p = 1/2 + \lebexp/(2q)$.

\ref{en:localpart_int}. Take any $q_0 \in (q,\infty)$ and $s_0 \in (1/q_0,\infty)$. Then
\[
\|M_\ell\|_\infty \lesssim \|M\|_\infty \lesssim \|M\|_{\AH^{s_0,-1/q_0}_{q_0,\infty}}
\]
by Young's inequality and Lemma \ref{lem:AHemb}\ref{en:AHemb_C2}, and therefore
\[
\|M_\ell(\DD_X)\|_{L^2(\mu_X) \to L^2(\mu_X)} \lesssim \|M\|_{\AH^{s_0,-1/q_0}_{q_0,\infty}}.
\]
On the other hand, if $s_1 = \sigma$ and $q_1 = \lebexp$, then, by Proposition \ref{prp:hardy_boundedness}\ref{en:localpart},
\[
\|M_\ell(\DD_X)\|_{\hardy(\mu_X) \to L^1(\mu_X)} \lesssim \|M\|_{\AH^{s_1,-1/q_1}_{q_1,\infty}}.
\]
Therefore, by interpolation (see \cite[Theorem 4.4.2]{BL}, Lemma \ref{lem:AHinterpol} and Theorem \ref{thm:hardy_interpol}),
\[
\|M_\ell(\DD_X)\|_{L^{p_\theta}(\mu_X) \to L^{p_\theta}(\mu_X)} \lesssim \|M\|_{\AH^{s_\theta,-1/q_\theta}_{q_\theta,\infty}}
\]
for all $\theta \in (0,1)$, where $s_\theta = (1-\theta)s_0+\theta s_1$ and $1/q_\theta = (1-\theta)/q_0+\theta/q_1$, $1/p_\theta = (1-\theta)/2+\theta/1 = 1/2 + \theta/2$.

If we take $\theta = \lebexp/q$, then $p_\theta = p$ and therefore it is enough to choose $q_0$ and $s_0$ so that the continuous embedding $\AH^{s,-1/q}_{q,\infty} \subseteq \AH^{s_\theta,-1/q_\theta}_{q_\theta,\infty}$ holds. Note, on the other hand, that $1/q_\theta > \theta/q_1 = 1/q$, that is, $q > q_\theta$; therefore, by parts \ref{en:AHemb_sigma} and \ref{en:AHemb_q1} of Lemma \ref{lem:AHemb}, the required embedding holds whenever $s \geq s_\theta$. Since $s_\theta = (1-\lebexp/q)s_0 + \sigma \lebexp/q$, and $s > \sigma \lebexp/q$ by assumption, the inequality  $s \geq s_\theta$ can be simply achieved by choosing $q_0 \in (q,\infty)$ sufficiently large and $s_0 \in (1/q_0,\infty)$ sufficiently small.

\bigskip
\ref{en:globalpart_int}. Take any $q_0 \in (q,\infty)$ and $s_0 \in (1/q_0,\infty)$. Then
\[
\|M_g\|_\infty \lesssim \|M\|_\infty \lesssim \|M\|_{\AH^{s_0,-1/q_0}_{q_0,\infty}} \lesssim \|M\|_{\AcH^{s_0,-1/q_0,0}_{q_0,q_0}}
\]
by Young's inequality and Lemma \ref{lem:AHemb}\ref{en:AHemb_C2}. Therefore
\[
\|M_g(\DD_X)\|_{L^2(\mu_X) \to L^2(\mu_X)} \lesssim \|M\|_{\AcH^{s_0,-1/q_0,0}_{q_0,q_0}}.
\]
On the other hand, if $s_1 > \thresh-1$ and $q_1 = \lebexp$, then, by Proposition \ref{prp:hardy_boundedness}\ref{en:globalpart},
\[
\|M_g(\DD_X)\|_{L^1(\mu_X) \to L^1(\mu_X)} \lesssim \|M\|_{\AcH^{s_1,\plnexp-s_1,W}_{q_1,q_1}}.
\]
Therefore, by interpolation (see \cite[Theorem 4.4.1]{BL} and Lemma \ref{lem:AcHinterpol}),
\[
\|M_g(\DD_X)\|_{L^{p_\theta}(\mu_X) \to L^{p_\theta}(\mu_X)} \lesssim \|M\|_{\AcH^{s_\theta,\tau_\theta,W\theta}_{q_\theta,q_\theta}}
\]
for all $\theta \in (0,1)$, where $s_\theta = (1-\theta)s_0+\theta s_1$, $\tau_\theta = -(1-\theta)/q_0 + \theta (\plnexp-s_1)$ and $1/q_\theta = (1-\theta)/q_0+\theta/q_1$, $1/p_\theta = (1-\theta)/2+\theta/1 = 1/2 + \theta/2$.

If we take $\theta = \lebexp/q$, then $p_\theta = p$ and $W\theta = W \lebexp/q$, and therefore it is enough to choose $q_0$ and $s_0$ so that the continuous embedding $\AH^{s,\tau-s}_{q,q} \subseteq \AH^{s_\theta,\tau_\theta}_{q_\theta,q_\theta}$ holds. Note, on the other hand, that $1/q_\theta > \theta/q_1 = 1/q$, that is, $q > q_\theta$; therefore, by parts \ref{en:AHemb_sigma}, \ref{en:AHemb_rtau} and \ref{en:AHemb_q1} of Lemma \ref{lem:AHemb}, the required embedding holds whenever $s \geq s_\theta$ and $\tau-s > \tau_\theta-1/q+1/q_\theta$. Since $\tau_\theta-1/q+1/q_\theta = \lebexp/q(\plnexp-s_1)$ and we have assumed $\tau > \plnexp \lebexp/q$ and $s > (\thresh+1) \lebexp/q$, the inequality $\tau-s > \tau_\theta-1/q+1/q_\theta$ can be rewritten as $\tau > \plnexp \lebexp/q + (s-s_1 \lebexp/q)$ and is achieved by choosing $s_1 \in (\thresh+1,s q/\lebexp)$ so that $s-s_1 \lebexp/q$ is sufficiently small. Therefore, since $s_\theta = (1-\theta)s_0 + s_1 \lebexp/q$ and $s > s_1 \lebexp/q$, the inequality $s \geq s_\theta$ is simply achieved by choosing $q_0 \in (q,\infty)$ sufficiently large and $s_0 \in (1/q_0,\infty)$ sufficiently small.

\bigskip
\ref{en:bothparts_int}. Note that
\[
\|M\|_{\AH^{s,-1/q}_{q,\infty}}  \lesssim \|M\|_{\AcH^{s,-1/q,W \lebexp/q}_{q,\infty}}
\]
by \eqref{eq:AcHthreelines}. Moreover, since $s > \plnexp \lebexp/q+1/q$, we can choose $\tau>\plnexp\lebexp/q$ so that $s > \tau+1/q$, whence $-1/q > \tau-s$ and therefore, by Lemma \ref{lem:AHemb}\ref{en:AHemb_rtau},
\[
\|M\|_{\AcH^{s,\tau-s,W \lebexp/q}_{q,q}} \lesssim \|M\|_{\AcH^{s,-1/q,W \lebexp/q}_{q,\infty}}.
\]
Since $M = M_\ell + M_g$, part \ref{en:bothparts_int} follows by combining \ref{en:localpart_int} and \ref{en:globalpart_int}.
\end{proof}

\section{Applications}\label{s:applications}

Here we show how our general conditional results, namely Theorem \ref{teo:main} and Corollary \ref{cor:main}, can be applied to refine the multiplier theorem of \cite{HMM} for sub-Laplacians with drift on Lie groups of polynomial growth. Moreover we present an application to distinguished sub-Laplacians with drift on certain groups of exponential growth, whose no-drift counterpart was discussed in \cite{MOV}.

\subsection{Lie groups of polynomial growth}
Suppose that $G$ is a non-compact Lie group of polynomial growth.
Let $d_0,d_\infty \in \NN \setminus \{0\}$ and $\delta \in (0,1]$ be as in Proposition \ref{prp:basic_sub_riemannian}.
Set $d=\max\{d_0,d_{\infty}\}$. 

Let $\psi$ be defined as in \eqref{eq:sumpsi}.  For all $s \geq 0$ we define 
\[
\|F\|_{L^\infty_{s,\sloc}} = \sup_{t > 0} \|F(t\cdot) \, \psi\|_{\AH_\infty^s}.
\]
Finiteness of $\|F\|_{L^\infty_{s,\sloc}}$ can be thought of as a homogeneous $L^\infty$ Mihlin--H\"ormander condition, which can be compared with the $L^q$ inhomogeneous conditions introduced in Section \ref{s:smooth} by observing that, by Sobolev's embedding (see, e.g., \cite[Theorems 6.2.4 and 6.5.1]{BL} and \cite[Lemma 4.8]{martini_analysis_2012}), 
\begin{equation}\label{eq:Linftycomp}
\|F\|_{L^\infty_{s,\sloc}} \lesssim \| F \|_{\AH^{\sigma,-1/q}_{q,\infty}}
\end{equation}
whenever $q \in (1,\infty)$ and $\sigma > s+1/q$.

It is well-known \cite{A,cowling_spectral_2001,DOS,He} that a multiplier theorem of Mihlin--H\"ormander type holds for the sub-Laplacian without drift $\Delta$, entailing the weak-type estimate
\begin{equation}\label{eq:mhteo_poly}
\| F(\Delta) \|_{L^1(\mu) \to L^{1,\infty}(\mu)} \lesssim \|F\|_{L^\infty_{s,\sloc}}
\end{equation}
for all $s > d/2$. From its proof the following estimates can be derived.

\begin{lem}\label{lem:poly_MH}
Let $F: \RR \to \CC$ be a function such that $\|F\|_{L^\infty_{s,\sloc}} <\infty$ for some $s > d/2$.
 Then the following inequalities hold:
\begin{enumerate}[label=(\roman*)]
\item\label{en:poly_MH_diff} $\sup_{y \in G} \int_{|x|_\dist \geq 2|y|_\dist} |k_{F(\Delta)}(xy)-k_{F(\Delta)}(x)| \di\mu (x) \leq C_{s} \| F \|_{L^\infty_{s,\sloc}}$;
\item\label{en:poly_MH_annulus} $\sup_{r>0} \int_{r \leq |x|_\dist < 2r} |k_{F(\Delta)}(x)| \di\mu(x) \leq C_{s} \| F \|_{L^\infty_{s,\sloc}}$;
\end{enumerate} 
In addition, if $r \in (0,\infty)$, $F : \RR \to \CC$ is even and $\hat F \subseteq [-r,r]$, then, for all $\varepsilon>0$,
\begin{enumerate}[label=(\roman*),resume]
\item\label{en:poly_fps_plancherel} $\|\chi^{1/2} k_{F(\DD)}\|_{L^1(\mu)} \leq C_\varepsilon (1+r)^{(d_\infty-\delta)/2} e^{b_X r} \sup_{t \geq 0} |(1+t)^{d_0/2+\varepsilon} \, F(t)|$.\end{enumerate}
\end{lem}
\begin{proof}
By replacing $F$ with $\overline{F}$, exploiting the unimodularity of $G$ and the fact that $k_{\overline{F}(\Delta)}(x) = \overline{k_{F(\Delta)}(x^{-1})}$, the estimate in part \ref{en:poly_MH_diff} can be equivalently rewritten as
\[
\sup_{y \in G} \int_{|x|_\dist \geq 2|y|_\dist} |k_{F(\Delta)}(y^{-1}x)-k_{F(\Delta)}(x)| \di\mu(x) \leq C_{s} \| F \|_{L^\infty_{s,\sloc}},
\]
which is proved in \cite[formula (13)]{A}.

Let us now prove part \ref{en:poly_MH_annulus}. Note that, by Proposition \ref{prp:basic_sub_riemannian},
\begin{equation}\label{eq:poly_doublingQ}
V_\dist(\lambda r) \leq C (1+\lambda)^d V_\dist(r)
\end{equation}
for all $\lambda,r \in \Rpos$. Moreover, by \cite[Theorem VIII.2.4]{varopoulos_analysis_1992}, $\Delta$ satisfies Gaussian-type heat kernel bounds of the form
\[
|k_{\exp(-t\Delta)}| \leq C \, V_\dist(t^{1/2})^{-1} \exp(- b|x|_\dist^2/t)
\]
for some $C,b \in \Rpos$ and all $t \in \Rpos$. Hence we can apply \cite[Theorem 6.1]{M} to the operator $\Delta$.

Define now
\[
F_j(\lambda)=F(2^j\lambda)\,\psi(\lambda)\qquad\forall j\in\ZZ \quad \forall\lambda\in\Rpos
\]
and set $\tilde F_J = 
\sum_{j < J} F_j(2^{-j} \cdot)$ for all $J \in \ZZ$.
Then, for all $J \in \ZZ$,
\[
F(\Delta)= \tilde F_J(\Delta) + \sum_{\substack{j\in\ZZ \\ j \geq J}} F_j(2^{-j}\Delta)\,,
\]
in the sense of strong convergence of operators on $L^2(\mu)$.

Choose $\varepsilon \in \Rpos$ so that  $s > d/2+\varepsilon$.  Fix $r \in \Rpos$, and let $J \in \ZZ$ be minimal so that $2^{J/2} r > 1$. 

Note that $\supp F_j \subseteq [2^{j-2},2^{j+2}]$. By \cite[Theorem 6.1(ii)]{M}, for all $j \in \ZZ$ such that $j \geq J$,
\[\begin{split}
\int_{r \leq |x|_\dist < 2r} |k_{F_j(2^{-j}\Delta)}(x)| \di \mu(x) 
&\leq (1+2^{1+j/2} r)^{-\varepsilon} \|(1+2^{1+j/2} |\cdot|_\dist)^\varepsilon k_{F_j(2^{-j}\Delta)}  \|_1 \\
&\leq C_{s} (1+2^{1+j/2} r)^{-\varepsilon} \|F_j\|_{\AH_\infty^s}  \\
&\leq C_{s} (2^{j/2} r)^{-\varepsilon} \|F\|_{L^\infty_{s,\sloc}}.
\end{split}\]

On the other hand, $\supp \tilde F_J \subseteq [0,2^{J+2}]$. Hence, by H\"older's inequality and \cite[Theorem 6.1(i)]{M},
\[\begin{split}
\int_{r \leq |x|_\dist < 2r} |k_{\tilde F_J(\Delta)}(x)| \di \mu(x) 
&\leq V_\dist(2r)^{1/2} \|k_{\tilde F_J(\Delta)}\|_2  \\
&\leq C V_\dist(2r)^{1/2} V_\dist(2^{-1-J/2})^{-1/2} \|\tilde F_J\|_\infty   \\
&\leq C  \|F\|_{\infty};
\end{split}\]
in the last step the fact that $2^{-J/2} \sim r$ and the doubling condition \eqref{eq:poly_doublingQ} were used.

Therefore, by summing the previous estimates,
\[\begin{split}
&\int_{r \leq |x| < 2r} |k_{F(\Delta)}(x)| \di \mu(x) \\
&\leq \int_{r \leq |x| < 2r} |k_{\tilde F_J(\Delta)}(x)| \di \mu(x) + \sum_{\substack{j \in \ZZ \\ j \geq J }} \int_{r \leq |x| < 2r} |k_{F_j(2^{-j} \Delta)}(x)| \di \mu(x) \\
&\leq C_s \left(\|F\|_\infty +  \|F\|_{L^\infty_{s,\sloc}} \sum_{\substack{j \in \ZZ \\ 2^{j/2}r >1 }} (2^{j/2} r)^{-\varepsilon}  \right) \\
&\leq C_s  \|F\|_{L^\infty_{s,\sloc}}
\end{split}\]
and part \ref{en:poly_MH_annulus} follows.

As for part \ref{en:poly_fps_plancherel}, note that, if $\supp \hat F \subseteq [-r,r]$, then, by finite propagation speed,
$\supp k_{F(\DD)} \subseteq B_\dist(e,r)$
and therefore, by Lemma \ref{lem:char_estimate} and the Cauchy--Schwarz inequality,
\[
\|\chi^{1/2} k_{F(\DD)}\|_{L^1(\mu)} \lesssim (1+r)^{(d_\infty-\delta)/2} e^{b_X r} \| k_{F(\DD)} \|_{L^2(\mu)}.
\]
On the other hand, for all $\varepsilon > 0$,
\[\begin{split}
\| k_{F(\DD)} \|_{L^2(\mu)} 
&= \| F(\DD) \|_{L^2(\mu) \to L^\infty(\mu)} \\
&\leq \| (1+\Delta)^{(d_0/2+\varepsilon)/2} \, F(\DD) \|_{L^2(\mu) \to L^2(\mu)} \\
&\times \| (1+\Delta)^{-(d_0/2+\varepsilon)/2} \|_{L^2(\mu) \to L^\infty(\mu)}
\end{split}\]
and, by spectral theory,
\[
\| (1+\Delta)^{(d_0/2+\varepsilon)/2} \, F(\DD) \|_{L^2(\mu) \to L^2(\mu)} \lesssim \sup_{t \geq 0} | (1+t)^{d_0/2+\varepsilon} \, F(t)|.
\]
Since $\| (1+\Delta)^{-(d_0/2+\varepsilon)/2} \|_{L^2(\mu) \to L^\infty(\mu)} < \infty$ (see, e.g., \cite[proof of Lemma 5.6]{HMM}), part \ref{en:poly_fps_plancherel} follows.
\end{proof}

The previous lemma, together with \eqref{eq:Linftycomp} and Remarks \ref{rem:CZ} and \ref{rem:linfty_lq}, shows that Assumptions \ref{en:ass_MH_diff}, \ref{en:ass_MH_annulus}, \ref{en:ass_fps_plancherel} are satisfied for $W = b_X$, $\lebexp \in [2,\infty)$, $\thresh = (d_\infty-\delta)/2$, $\plnexp > d_0/2$, $\sigma > \max\{d_0,d_\infty\}/2$.
Hence in this case from Corollary \ref{cor:main} we obtain the following multiplier theorem, whose part \ref{en:main_lppol} refines \cite[Theorem 5.2]{HMM}.

\begin{teo}\label{teo:main-pol}
Suppose that $G$ is a non-compact Lie group of polynomial growth, and let $d_0,d_\infty \in \NN \setminus \{0\}$ and $\delta \in (0,1]$ be as in Proposition \ref{prp:basic_sub_riemannian}.
Let $p \in [1,\infty] \setminus \{2\}$. Let $W_{X,p}$ be defined as in \eqref{eq:strip_width}.
Suppose that $M\in H^{\infty}(\Sigma_{W_{X,p}};s)$ for some $s \in \NN$,
\begin{equation}\label{eq:poly_smoothness}
s > |1/p-1/2| \max\{ d_0, d_\infty-\delta+2 \} .
\end{equation}
Then the following hold:
\begin{enumerate}[label=(\roman*)]
\item\label{en:main_lppol} if $p\in (1,\infty)$, then $M(\DD_X)$ extends to a bounded operator on $L^p(\mu_{X})$;
\item\label{en:main_h1pol} if $p=1$, then $M(\DD_X)$ extends to a bounded operator from $\hardy(\mu_{X})$ to $L^1(\mu_{X})$;
\item\label{en:main_bmopol} if $p=\infty$, then $M(\DD_X)$ extends to a bounded operator from $L^{\infty}(\mu_{X})$ to $\bmo(\mu_{X})$. 
\end{enumerate}
\end{teo}

Note that in \cite[Theorem 5.2]{HMM} the condition $s>\max\{d_0+4,d_\infty+2\}/2$ is required instead of \eqref{eq:poly_smoothness}.

Suppose now that $G$ is a stratified group and $\Delta$ is a homogeneous sub-Laplacian thereon. In this case $d = d_0 = d_\infty$ is the homogeneous dimension of $G$, and moreover $\delta = 1$ (since $V_\dist(r) = V_\dist(1) \, r^d$), so the condition \eqref{eq:poly_smoothness} becomes
\[
s > |1/p-1/2| (d+1) .
\]
This condition (or rather its endpoint version for $p=1$) should be compared with the condition $s > d/2$ in the multiplier theorem \eqref{eq:mhteo_poly} for the sub-Laplacian without drift $\Delta$. Note that, in the case of stratified groups, a sharper result for $\Delta$ is available \cite{christ_multipliers_1991,mauceri_vectorvalued_1990}, where the $L^\infty$ Sobolev norm in \eqref{eq:mhteo_poly} is replaced by an $L^2$ Sobolev norm.
Moreover, for many classes of stratified groups (especially $2$-step groups), the condition $s>d/2$ is not optimal and in a number of cases it can be pushed down to $s > (\dim G)/2$, where $\dim G$ is the topological dimension of $G$ \cite{hebisch_multiplier_1993,martini_heisenbergreiter,MMue,MMue2,MueS}. It is conceivable that similar techniques could be adapted to sharpen the estimates in Lemma \ref{lem:poly_MH} and refine the result for the sub-Laplacian with drift $\Delta_X$ on those classes of groups.

\subsection{Solvable extensions of stratified groups}\label{Subsec:NA}
In this subsection we prove a multiplier theorem for the sub-Laplacian with drift $\Delta_X$ when $G$ is a solvable extension of a stratified Lie group constructed as follows (we refer to \cite{MOV} for more details on the construction). 
 
Let $N$ be a stratified group. In other words, $N$ is a simply connected Lie group, whose Lie algebra $\lie{n}$ is endowed with a derivation $D$ such that the eigenspace of $D$ corresponding to the eigenvalue $1$ generates $\lie{n}$ as a Lie algebra. 
The eigenvalues of $D$ are positive integers $1,\dots,\step$ and $\lie{n}$ is the direct sum of the eigenspaces of $D$, which are called layers: the $j$th layer corresponds to the eigenvalue $j$. Moreover $\lie{n}$ is $\step$-step nilpotent, where $\step$ is the maximum eigenvalue.
 The exponential map $\exp_N : \lie{n} \to N$ is a diffeomorphism and provides global coordinates for $N$. Any chosen Lebesgue measure on $\lie{n}$ is then a left and right Haar measure on $N$.  

The formula $\delta_t = \exp((\log t) D)$ defines a family of automorphic dilations $(\delta_t)_{t>0}$ on $N$. For all measurable sets $E \subset N$ and $t > 0$, $|\delta_t E| = t^Q |E|$, where $Q = \tr D$ is the homogeneous dimension of $N$. Note that $Q \geq \dim N$, where $\dim N$ is the topological dimension of $N$, and in fact $Q = \dim N$ if and only if $\step=1$, i.e., if and only if $N$ is abelian. Note moreover that, if $Q = 1$, then $N \cong \RR$.

Let $A = \RR$, considered as an abelian Lie group. Again we identify $A$ with its Lie algebra $\lie{a}$. Then $A$ acts on $N$ by dilations, that is, we have a homomorphism $A \ni u \mapsto \delta_{e^u} \in \Aut(N)$ and we can define the corresponding semidirect product $G = N \rtimes A$, with operation
\[
(z,u) \cdot (z',u') = (z \cdot e^{uD} z', u+u') \,.\]
The group $G$ is a solvable Lie group. The right Haar measure $\mu$ on $G$ is given by
\[
\di\mu(z,u) = \di z \di u
\]
\cite[\S (15.29)]{hewitt_abstract_1979} and the modular function $m$ is given by $m(z,u) =  e^{-Qu}$. In particular $G$ is not unimodular and has exponential volume growth \cite[Lemme I.3]{guivarch_croissance_1973}.

Consider a system $\vfN_1,\dots,\vfN_\nu$ of left-invariant vector fields on $N$ that form a basis of the first layer of the Lie algebra of $N$. Let $\vfA = \partial_u$ be the canonical basis of $\lie{a}$. The vector fields $\vfA$ on $A$ and $\vfN_1,\dots,\vfN_\nu$ on $N$ can be lifted to left-invariant vector fields on $G$ given by
\[
\vfG_0|_{(z,u)} = \vfA|_z = \partial_u,  \qquad \vfG_j|_{(z,u)} = e^{u} \vfN_j|_z \qquad \text{ for $j=1,\dots,\nu.$}
\]
The system $\vfG_0,\dots,\vfG_\nu$ generates the Lie algebra $\lie{g}$ of $G$ and determines a sub-Riemannian structure and a left-invariant sub-Laplacian $\Delta=-\sum_{j=0}^{\nu}X_j^2$ on $G$. A formula for the sub-Riemannian distance $\dist$ is available in this case \cite[Proposition 2.7]{MOV} and the corresponding local dimension $d_0$ (see  Proposition \ref{prp:basic_sub_riemannian}) equals $Q+1$.

Recall from the discussion in Section \ref{s:preliminaries} that symmetric sub-Laplacians with drifts on $G$ correspond to positive characters.

\begin{lem}\label{l:characters}
All nontrivial positive characters of $G$ are of the form
\begin{equation}\label{eq:char_NA}
\chi_{\alpha}(z,u)=e^{\alpha u}, \qquad\forall (z,u)\in G,
\end{equation}
for some $\alpha\in \RR\setminus\{0\}$.
\end{lem}
\begin{proof}
Let $\chi$ be a positive character of $G$. Then its differential $\chi' : \lie{g} \to \RR$ is a homomorphism of Lie algebras and therefore $\chi'$ vanishes on $[\lie{g},\lie{g}]$. On the other hand, as explained in \cite[\S 2.1]{MOV}, $\lie{g}$ can be identified with the semidirect product $\lie{n}\rtimes \lie{a}$, and $[\lie{g},\lie{g}] = \{(z,0) \tc z \in \lie{n}\}$. Hence $\chi'(z,u) = \alpha u$ for some $\alpha \in \RR$, and consequently $\chi(z,u) = e^{\alpha u}$.
\end{proof}
Given $\alpha\in \RR\setminus\{0\}$, we consider the vector field $X=\alpha X_0$. It is easy to check that $X|_e=\nabla_H \chi_\alpha|_e$, where $\chi_\alpha$ is given by \eqref{eq:char_NA}. We shall prove a multiplier theorem for $\Delta_{X}=\Delta-X$. 

Let $\psi$ be defined as in \eqref{eq:sumpsi}. For all 
$s \geq 0$ 
we define
\[
\|F\|_{L^2_{s,\sloc}} = \sup_{t > 0} \|F(t\cdot)\,\psi\|_{\AH_2^{s}}.
\]
Finiteness of 
$\|F\|_{L^2_{s,\sloc}}$ 
can be thought of as a homogeneous Mihlin--H\"ormander condition, which can be compared with the inhomogeneous conditions introduced in Section \ref{s:smooth} by observing that
\begin{equation}\label{eq:L2comp}
\|F\|_{L^2_{s,\sloc}} \lesssim \| F \|_{\AH^{s,-1/2}_{2,\infty}}
\end{equation}
whenever $s  > 1/2$ (see, e.g., \cite[Lemma 4.8]{martini_analysis_2012}).

In \cite{MOV} a multiplier theorem of Mihlin--H\"ormander type for the sub-Laplacian without drift $\Delta$ was proved, entailing the weak-type estimate
\begin{equation}\label{eq:mhteo_na}
\|F(\Delta)\|_{L^1(\mu) \to L^{1,\infty}(\mu)} \lesssim \|F\|_{L^2_{s,\sloc}}
\end{equation}
for $s > \max\{3/2,(Q+1)/2\}$. From its proof a number of estimates can be derived, that can be used to verify the assumptions of the conditional multiplier theorem of Section \ref{s:main}.
It should be noted that \cite{MOV} only discusses the case $Q \geq 2$ explicitly, but a few small adjustments allow one to consider the case $Q=1$ as well.

For all $\lambda \in \Rpos$ and $a,b \in \RR$, let $\lambda^{[a,b]}$ denote $\lambda^a$ if $\lambda \leq 1$ and $\lambda^b$ if $\lambda \geq 1$.

\begin{lem}\label{lem:na_variation}
For all even bounded Borel functions $F : \RR \to \CC$,
\begin{equation}\label{eq:na_plancherel}
\|k_{F(\DD)}\|_{L^2(\mu)}^2 \sim \int_0^\infty |F(\lambda)|^2 \,\lambda^{[3,Q+1]} \, \frac{\ndi\lambda}{\lambda}.
\end{equation}
Moreover, if $\supp F \subseteq [-2,2]$ then, for all $\varepsilon \in \Rnon$, $t \in \Rpos$ and $y \in B_\dist(e,1)$,
\begin{align}
\label{eq:na_cptF} \int_{|x|_\dist \leq 4} |k_{F(t\DD)}(x)| (1+t^{-1} |x|_\dist)^\varepsilon \di\mu(x) &\leq C_{s,\varepsilon} \|F\|_{\AH_2^s}, \\
\label{eq:na_cptFgrad} \int_{|x|_\dist \leq 4} |\nabla_H k_{F(t\DD)}(x)| (1+t^{-1} |x|_\dist)^\varepsilon \di\mu(x) &\leq C_{s,\varepsilon} \, t^{-1} \|F\|_{\AH_2^s}, \\
\label{eq:na_cptFdiff} \int_{|x|_\dist \leq 3} |k_{F(t\DD)}(xy) - k_{F(t\DD)}(x)| \di\mu(x) &\leq C_{s} \, t^{-1} \, |y|_\dist \, \|F\|_{\AH_2^s}
\end{align}
whenever $t \leq 1$ and $s > (Q+1)/2+\varepsilon$, or $t \geq 1$ and $s > 0$.
\end{lem}
\begin{proof}
The Plancherel estimate \eqref{eq:na_plancherel} is just a rephrasing of \cite[Corollary 4.6]{MOV}, in a form that is valid in the case $Q=1$ too (cf.\ \cite[Introduction, Theorem 4.2]{Helg}).

To prove \eqref{eq:na_cptF}, note first that
 from finite propagation speed \eqref{eq:fps}, the Cauchy--Schwarz inequality and the fact that $Q+1$ is the local dimension associated to the sub-Riemannian structure, one immediately obtains that, for all $r \in \Rpos$ and all even $f : \RR \to \CC$ with $\supp \hat f \subseteq [-r,r]$,
\begin{equation}\label{eq:na_l1l2}
\int_{|x|_\dist \leq 4} |k_{f(\DD)}(x)| \di\mu(x) \lesssim \min\{r^{(Q+1)/2},1\} \, \|k_{f(\DD)}\|_{L^2(\mu)}.
\end{equation}
Take now an even $F : \RR \to \CC$ with $\supp F \subseteq [-2,2]$, and decompose $F(t \cdot) = \sum_{\ell=0}^\infty f_{\ell,t}$ as in \cite[Lemma 5.2]{MOV}. Similarly as in the proof of \cite[Proposition 5.3]{MOV}, but using \eqref{eq:na_l1l2} in place of \cite[Proposition 5.1]{MOV}, one then obtains
\[\begin{split}
&\int_{|x|_\dist \leq 4} |k_{f_{\ell,t}(\DD)}(x)| \, (1+t^{-1} |x|_\dist)^\varepsilon \di\mu(x) \\
&\lesssim (1+t^{-1}\min\{2^\ell t, 1\})^\varepsilon \min\{(2^\ell t)^{(Q+1)/2},1\} \| k_{f_{\ell,t(\DD)}} \|_{L^2(\mu)} \\
&\lesssim (1+t^{-1}\min\{2^\ell t, 1\})^\varepsilon \min\{(2^\ell t)^{(Q+1)/2},1\} \, 2^{-s\ell} \, (t^{-1})^{[3/2,(Q+1)/2]} \, \|F\|_{\AH_2^s}.
\end{split}\]
Hence, if $t \geq 1$, then
\[
\int_{|x|_\dist \leq 4} |k_{F(t\DD)}(x)| \, (1+t^{-1} |x|_\dist)^\varepsilon\di\mu(x) \lesssim (1+t^{-1})^\varepsilon \, t^{-3/2} \, \|F\|_{\AH_2^s} \sum_{\ell=0}^\infty 2^{-s\ell} \lesssim  \|F\|_{\AH_2^s}
\]
whenever $s > 0$. If instead $t \leq 1$, then
\[\begin{split}
&\int_{|x|_\dist \leq 4} |k_{F(t\DD)}(x)| \, (1+t^{-1} |x|_\dist)^\varepsilon\di\mu(x) \\
&\lesssim  \|F\|_{\AH_2^s} \left( t^{-(Q+1)/2-\varepsilon} \sum_{\ell \tc 2^\ell \geq t^{-1}} 2^{-s\ell}  +  \sum_{\ell \tc 2^\ell < t^{-1}} 2^{(\varepsilon+(Q+1)/2-s)\ell} \right) \lesssim \|F\|_{\AH_2^s}
\end{split}\]
whenever $s > (Q+1)/2+\varepsilon$.

As for \eqref{eq:na_cptFgrad}, observe that the following analogue of \eqref{eq:na_l1l2} holds for all $r \in \Rpos$ and all even $f : \RR \to \CC$ with $\supp \hat f \subseteq [-r,r]$:
\begin{equation}\label{eq:na_l1l2grad}
\int_{|x|_\dist \leq 4} |\nabla_H k_{f(\DD)}(x)| \di\mu(x) \lesssim \min\{r^{(Q+1)/2},1\} \, \| |\nabla_H k_{f(\DD)}| \|_{L^2(\mu)};
\end{equation}
moreover, by \eqref{eq:na_plancherel}, for all even $F : \RR \to \CC$,
\begin{equation}\label{eq:na_plancherelgrad}
\|| \nabla_H k_{F(\DD)}| \|_{L^2(\mu)}^2 = \|\DD k_{F(\DD)}\|_{L^2(\mu)}^2 \sim \int_0^\infty |F(\lambda)|^2 \, \lambda^{[5,Q+3]} \frac{\ndi\lambda}{\lambda}.
\end{equation}
By repeating the above argument for \eqref{eq:na_cptF}, but using \eqref{eq:na_l1l2grad} and \eqref{eq:na_plancherelgrad} in place of \eqref{eq:na_l1l2} and \eqref{eq:na_plancherel}, one easily obtains \eqref{eq:na_cptFgrad}.

Finally, \eqref{eq:na_cptFdiff} is an immediate consequence of \eqref{eq:na_cptFgrad} in the case $\varepsilon = 0$ (cf. \cite[Lemma 5.4]{MOV} and the proof of \cite[Lemma VIII.1.1]{varopoulos_analysis_1992}).
\end{proof}

\begin{lem}\label{lem:na_MH}
Let $F: \RR \to \CC$ be even and such that $\supp \hat F \subseteq [-2,2]$, and let $s> \frac{Q+1}{2}$.
Then the following inequalities hold:
\begin{enumerate}[label=(\roman*)]
\item\label{en:na_MH_diff} $\sup_{y\in B_\dist(e,1)} \int_{|x|_\dist \geq 2 |y|_\dist} |k_{F(\DD)}(xy)-k_{F(\DD)}(x)| \di\mu(x) \leq C_{s} \| F \|_{L^2_{s,\sloc}}$;
\item\label{en:na_MH_annulus} $\sup_{0<r\leq 1} \int_{r \leq |x|_\dist < 2r} |k_{F(\DD)}(x)| \di\mu(x) \leq C_{s} \| F \|_{L^2_{s,\sloc}}$.
\end{enumerate} 
In addition, for all $r \in (0,\infty)$ and all even $F : \RR \to \CC$ such that $\hat F \subseteq [-r,r]$, the estimate
\begin{enumerate}[label=(\roman*),resume]
\item\label{en:na_fps_plancherel} $\|\chi_\alpha^{1/2} k_{F(\DD)}\|_1 \leq C_\alpha (1+r)^\thresh \, e^{r |\alpha|/2} (\int_0^\infty |F(\lambda)|^2 \lambda^{[3,Q+1]} \frac{\ndi\lambda}{\lambda})^{1/2}$
\end{enumerate}
holds with $\thresh = 1$, and actually one can take $\thresh = 1/2$ when $\alpha= -Q/2$, and $\thresh=0$ when $\alpha < -Q/2$.
\end{lem}
\begin{proof}
We first prove parts \ref{en:na_MH_diff} and \ref{en:na_MH_annulus}. Let $\psi$ be defined as in \eqref{eq:sumpsi}. Choose $\varepsilon \in(0,1)$ such that
$s>\frac{Q+1}{2}+\varepsilon$. Arguing as in the proof of \cite[Theorem 1.1]{MOV} define
\[
F_j(\lambda)=F(2^j\lambda)\,\psi(\lambda)\qquad\forall j\in\ZZ \quad \forall\lambda\in\Rpos.
\]
Then
\[
F(\DD)=\sum_{j\in\ZZ} F_j(2^{-j}\DD)\,,
\]
in the sense of strong convergence of operators on $L^2(\mu)$.

Let $k_j = k_{F_j(2^{-j}\DD)}$.
Hence, from \eqref{eq:na_cptF} and \eqref{eq:na_cptFdiff}
it follows that, for all $j \in \ZZ$ and $y \in B_\dist(e,1)$,
\begin{equation}\label{eq:stimaL1pesata}
\int_{|x|_\dist \leq 4} |k_j(x)|\,(1+2^{j}|x|_\dist)^{\varepsilon} \di\mu(x) \lesssim  \| F \|_{L^2_{s,\sloc}},
\end{equation}
and
\begin{equation}\label{eq:stimadifferenza}
\int_{|x|_\dist \leq 3} |k_j(xy)-k_j(x)| \di\mu(x)
\lesssim 
2^{j}|y|_\dist \| F \|_{L^2_{s,\sloc}}.
\end{equation}
Take any $y \in B_\dist(e,1)$ and choose $J$ as the smallest integer such that $2^{J}|y|_\dist>1$. Then, by \eqref{eq:stimaL1pesata}, 
\begin{equation}\label{eq:j>J}
\begin{split}
\sum_{j> J}\int_{3 \geq |x|_\dist \geq 2|y|_\dist}|k_{j}(xy)-k_{j}(x)| \di\mu(x)
&\leq 2 \sum_{j> J}\int_{4 \geq |x|_\dist \geq |y|_\dist}|k_{j}(x)| \di\mu(x) \\
&\lesssim \sum_{j > J} \big(1+2^{j} |y|_\dist\big)^{-\varepsilon} \| F \|_{L^2_{s,\sloc}}\\
&\lesssim \| F \|_{L^2_{s,\sloc}} .
\end{split}
\end{equation}
Moreover, by \eqref{eq:stimadifferenza},
\begin{equation}\label{eq:j<J}
\begin{split}
\sum_{j\leq J}\int_{3 \geq |x|_\dist \geq 2|y|_\dist} |k_{j}(xy)-k_{j}(x)| \di \mu (x)
&\lesssim  \sum_{j\leq J} 2^{j} |y|_\dist \, \| F \|_{L^2_{s_0,s_{\infty},\sloc}}\\
&\lesssim 2^{J}  |y|_\dist \, \| F \|_{L^2_{s,\sloc}}\\
&\lesssim \| F \|_{L^2_{s,\sloc}}.
\end{split}
\end{equation}
Note also that $\supp k_{F(\DD)} \subseteq B_\dist(e,2)$ by finite propagation speed \eqref{eq:fps} and the condition $\supp F \subseteq [-2,2]$, so the integral in the left-hand side of \ref{en:na_MH_diff} is actually restricted to $|x|_\dist \leq 3$. Hence the estimate \ref{en:na_MH_diff} follows by \eqref{eq:j>J} and \eqref{eq:j<J}. 

We now prove \ref{en:na_MH_annulus}. Fix $r \in (0,1]$.
If we choose $J$ as the smallest integer such that $2^{J}r >1$, then \eqref{eq:stimaL1pesata} implies that 
\begin{equation}\label{eq:j>I}
\sum_{j\geq J}\int_{r\leq|x|_\dist<2r}|k_j(x)| \di\mu(x)\lesssim  r^{-\varepsilon} \sum_{j\geq J} 2^{-\varepsilon j} \| F \|_{L^2_{s,\sloc}} \lesssim \| F \|_{L^2_{s,\sloc}} .
\end{equation}
Take now $j<J$. 
Then, by the Cauchy--Schwarz inequality, \cite[eq.\ (2.11)]{MOV} and the Plancherel estimate \eqref{eq:na_plancherel},
\[\begin{split}
\int_{r\leq|x|_\dist<2r}|k_j(x)| \di\mu(x) &\lesssim r^{(Q+1)/2} \| k_j\|_2 \\
&\lesssim r^{(Q+1)/2} (2^{j})^{[3/2,(Q+1)/2]} \|F_j\|_2 \\
&\lesssim  (r 2^{j})^{\delta} \, \| F_j \|_{\AH_2^s},
\end{split}\]
where $\delta = \min\{3/2,(Q+1)/2\} > 0$. Therefore
\begin{equation}\label{eq:j<Ismall}
\sum_{j< J}\int_{r\leq|x|_\dist<2r}|k_j(x)| \di\mu(x) \lesssim \sum_{j\leq J} (r 2^{j})^\delta \| F \|_{L^2_{s,\sloc}}  \lesssim \| F \|_{L^2_{s,\sloc}} .
\end{equation}
The estimate \ref{en:na_MH_annulus} follows from \eqref{eq:j>I} and \eqref{eq:j<Ismall}.

As for part \ref{en:na_fps_plancherel}, we follow the proof of \cite[Proposition 5.1]{MOV}: if $w : G \to \Rnon$ is the weight defined there, then, by finite propagation speed and H\"older's inequality,
\[
\|\chi_\alpha^{1/2} k_{F(\DD)}\|_1 \leq \left(\int_{B_\dist(e,r)} (1+w)^{-1} \chi_\alpha \di\mu \right)^{1/2} \left( \|k_{F(\DD)}\|_2 + \|w^{1/2} k_{F(\DD)}\|_2 \right).
\]
The second factor is estimated as in \cite[proof of Proposition 5.1]{MOV}:
\[
\|k_{F(\DD)}\|_2 + \|w^{1/2} k_{F(\DD)}\|_2 
\lesssim (1+r)^{1/2} \|k_{F(\DD)}\|_2.
\]
As for the first factor, following the proof of \cite[eq.\ (2.13)]{MOV},
\[\begin{split}
\int_{B_\dist(e,r)} (1+w)^{-1} \chi_\alpha \di\mu &
\sim \int_{-r}^r e^{\alpha u} \int_0^{e^u (\cosh r - \cosh u)}  \frac{s^{Q/2-1}}{1+s^{Q/2}} \di s \di u \\
&\lesssim (1+r) \int_0^r e^{|\alpha| u} \di u \lesssim (1+r) \, e^{|\alpha| r}.
\end{split}\]
Hence, by combining the two estimates,
\[
\|\chi_\alpha^{1/2} k_{F(\DD)}\|_1 \lesssim (1+r) \, e^{r|\alpha|/2} \| k_{F(\DD)} \|_2,
\]
and part \ref{en:na_fps_plancherel} with $\thresh = 1$ follows by \eqref{eq:na_plancherel}.

In order to improve the estimate in the case $\alpha \leq -Q/2$, we apply H\"older's inequality to obtain
\[
\|\chi_\alpha^{1/2} k_{F(\DD)}\|_1 \leq \left(\int_{B_\dist(e,r)} \chi_\alpha \di \mu \right)^{1/2} \|k_{F(\DD)}\|_2,
\]
and observe that, since $\alpha \leq -Q/2$,
\[\begin{split}
\int_{B_\dist(e,r)} \chi_\alpha \di \mu &
\sim \int_{-r}^r e^{\alpha u} \int_0^{e^u (\cosh r - \cosh u)} s^{Q/2-1} \di s \di u \\
&\sim \int_{-r}^r e^{(\alpha+Q/2) u} (\cosh r - \cosh u)^{Q/2} \di u \\
&\lesssim e^{r Q/2} \int_0^r e^{-(\alpha+Q/2) u}  \di u \sim \begin{cases}
e^{|\alpha| r} & \text{if $\alpha < -Q/2$,} \\
r e^{|\alpha| r} & \text{if $\alpha = -Q/2$.}
\end{cases}
\end{split}\]
Hence the improved estimate \ref{en:na_fps_plancherel} follows, as before, by \cite[Corollary 4.6]{MOV}.
\end{proof}

The previous lemma, together with \eqref{eq:L2comp} and Remark \ref{rem:CZ},
shows that Assumptions \ref{en:ass_MH_diff}, \ref{en:ass_MH_annulus}, \ref{en:ass_fps_plancherel} are satisfied for $W = b_X = |\alpha|/2$, $\lebexp = 2$, $\sigma > (Q+1)/2$, $\plnexp=Q/2$, and $\thresh = 1$ (and in fact one can take $\thresh = 1/2$ if $\alpha = -Q/2$ and $\thresh = 0$ for $\alpha < -Q/2$).
Hence from Corollary \ref{cor:main} we deduce the following multiplier theorem. 

\begin{teo}\label{teo:main_na}
Suppose that the group $G=N \rtimes A$ and the sub-Laplacian $\Delta = -\sum_{j=0}^\nu X_j^2$ are constructed as above, and let $Q$ be the homogeneous dimension of $N$.
Let $p\in [1,\infty]\setminus \{2\}$, $\alpha\in \mathbb R\setminus\{0\}$, $\chi=\chi_{\alpha}$ and $X=\alpha X_0$. Let $W_{X,p}$ be defined as in \eqref{eq:strip_width}. Suppose that $M\in H^{\infty}(\Sigma_{W_{X,p}};s)$ for some $s \in \NN$,
\begin{equation}\label{eq:smoothness_na}
s > |1/p-1/2| \max\{ Q+1, 3+\sgn(\alpha+Q/2) \}.
\end{equation}
Then the following hold.
\begin{enumerate}[label=(\roman*)]
\item\label{en:main_na_lp} If $p\in (1,\infty)$, then $M(\DD_X)$ extends to a bounded operator on $L^p(\mu_{X})$.
\item\label{en:main_na_h1} If $p=1$, then $M(\DD_X)$ extends to a bounded operator from $\hardy(\mu_{X})$ to $L^1(\mu_{X})$.
\item\label{en:main_na_bmo} If $p=\infty$, then $M(\DD_X)$ extends to a bounded operator from $L^{\infty}(\mu_{X})$ to $\bmo(\mu_{X})$. 
\end{enumerate}
\end{teo}

Actually, by Theorem \ref{teo:main}, the pointwise condition $M\in H^{\infty}(\Sigma_{W_{X,p}};s)$ in Theorem \ref{teo:main_na} can be replaced by the weaker $L^q$-type condition $M \in \AcH^{s,-1/q,W_{X,p}}_{q,\infty}$, where $1/q = |1/2-1/p|$, and in that case the order of smoothness $s$ need not be an integer.

Note that in many cases (e.g., when $Q > 2$ or $\alpha < -Q/2$) the condition \eqref{eq:smoothness_na} simply reduces to
\[
s > |1/p-1/2| \, (Q+1).
\]
This includes the case where $\alpha=-Q$: in this case, $\chi_\alpha=m$, so $\mu_X$ is a left Haar measure on $G$ and $\Delta_X$ is the ``intrinsic hypoelliptic Laplacian'' on $G$ \cite{ABGR}. 
If moreover $N$ is abelian, then $G$ and $\Delta_X$ can be identified with a rank-one Riemannian symmetric space and its Laplace--Beltrami operator, and $Q+1$ is the topological dimension of $G$. Hence, in this particular case, Theorem \ref{teo:main_na}\ref{en:main_na_lp} reduces to \cite[Theorem 1]{anker}.

\acknowledgments


\bigskip 
 
 
\begin{thebibliography}{55}




\bibitem{ABGR}
A.~A. Agrachev, U. Boscain, J.-P. Gauthier, and F. Rossi,
\emph{The intrinsic hypoelliptic {L}aplacian and its heat kernel on unimodular {L}ie groups},
J. Funct. Anal. \textbf{256} (2009), 2621--2655.



\bibitem{A} G. Alexopoulos, \emph{Spectral multipliers on Lie groups of polynomial growth}, Proc. Amer. Math. Soc. \textbf{120} (1994), 973--979.

\bibitem{A2} G. Alexopoulos, \emph{Sub-Laplacians with drift on Lie groups of polynomial volume growth}, Mem. Amer. Math. Soc. \textbf{155} (2002), no. 739.

\bibitem{anker} J.-P. Anker, \emph{$L_p$ Fourier multipliers on Riemannian symmetric spaces of the noncompact type}, Ann. Math. (2) \textbf{132} (1990), 597--628.

\bibitem{anker_multiplicateurs_1986}
J.-P. Anker and N. Lohou\'e, \emph{Multiplicateurs sur certains espaces sym\'etriques},
Amer. J. Math. \textbf{108} (1986), 1303--1353.


\bibitem{BL} J. Bergh and J. L\"ofstr\"om, \emph{Interpolation spaces. An introduction}, Springer, Berlin-New York, 1976.

\bibitem{Bourbaki} N. Bourbaki, \emph{General topology. {C}hapters 1--4}, Addison--Wesley, Reading, MA, 1966.

\bibitem{breuillard} E. Breuillard, \emph{Geometry of locally compact groups of polynomial growth and shape of large balls}, Groups Geom. Dyn. \textbf{8} (2014), 669--732.

\bibitem{BBI} D. Burago, Y. Burago, and S. Ivanov, \emph{A course in metric geometry}, American Mathematical Society, Providence, RI, 2001.


     
\bibitem{CD} A. Carbonaro and O. Dragi\v cevi\' c, \emph{Functional calculus for generators of symmetric contraction semigroups}, Duke Math. J. \textbf{166} (2017), 937--974.

\bibitem{CMM1} A. Carbonaro, G. Mauceri, and S. Meda, \emph{$H^1$ and $BMO$ for certain locally doubling metric measure spaces}, Ann. Sc. Norm. Super. Pisa Cl. Sci. (5) \textbf{8} (2009), 543--582.

\bibitem{CMM2} A. Carbonaro, G. Mauceri, and S. Meda, \emph{Comparison of spaces of Hardy type for the Ornstein--Uhlenbeck operator}, Potential Anal. \textbf{33} (2010), 85--105.

\bibitem{CGT} J. Cheeger, M. Gromov, and M. Taylor, \emph{Finite propagation speed, kernel estimates for functions of the Laplace operator, and the geometry of complete Riemannian manifolds}, J. Differential Geom. \textbf{17} (1982), 15--53.

\bibitem{clerc_lp_1974}
J. L. Clerc and E. Stein, \emph{$L^p$-multipliers for noncompact symmetric spaces},
Proc. Natl. Acad. Sci. USA \textbf{71} (1974), 3911--3912.

\bibitem{christ_multipliers_1991}
M. Christ, \emph{{$L^p$} bounds for spectral multipliers on nilpotent groups},
Trans. Amer. Math. Soc. \textbf{328} (1991), 73--81.


\bibitem{coldingminicozzi}
T. H. Colding and W. P. Minicozzi II,
\emph{Liouville theorems for harmonic sections and applications},
Comm. Pure Appl. Math. \textbf{51} (1998), 113--138.


\bibitem{CDMY} M.~G. Cowling, I. Doust, A. McIntosh, and A. Yagi,
\emph{Banach space operators with a bounded $H^\infty$ functional calculus},
J. Austral.  Math. Soc. (A) \textbf{60} (1996), 51--89.


\bibitem{CGHM} M.~G. Cowling, S. Giulini, A. Hulanicki, and G. Mauceri,
                     \emph{Spectral multipliers for a distinguished Laplacian on certain groups of exponential growth},
                     Studia Math. \textbf{111} (1994), 103--121.

\bibitem{cowling_estimates_1993} M.~G. Cowling, S. Giulini, and S. Meda,
                     \emph{$L^p-L^q$ estimates for functions of the Laplace-Beltrami operator on noncompact symmetric spaces. I},
                     Duke Math. J. \textbf{72} (1993), 109--150.


\bibitem{cowling_subfinsler_2013}
M.~G. Cowling and A.~Martini, \emph{Sub-{F}insler geometry and finite
  propagation speed}, in: {Trends in Harmonic Analysis}, Springer, Milan, 2013,
  pp.~147--205.

\bibitem{cowling_spectral_2001}
M.~G. Cowling and A.~Sikora, \emph{A spectral multiplier theorem for a sub-Laplacian on $\mathrm{SU}(2)$},
 Math. Z. \textbf{238} (2001), 1--36.



\bibitem{D}
N.~Dungey,
\emph{Heat kernel and semigroup estimates for sub-Laplacians with drift on {L}ie
  groups},  {Publ. Mat.} \textbf{49} (2005), 375--391.





\bibitem{DOS}
X. T. Duong, E. M. Ouhabaz, and A. Sikora,
Plancherel-type estimates and sharp spectral multipliers,
{\em J. Funct. Anal.}
\textbf{196} (2002),
443--485.



\bibitem{giulini_lp_1997}
S. Giulini, G. Mauceri, and S. Meda, \emph{$L^p$ multipliers on noncompact symmetric spaces},
J. reine angew. Math. \textbf{482} (1997), 151--175.

\bibitem{G} D. Goldberg, \emph{A local version of real Hardy spaces}, Duke Math. J. \textbf{46} (1979), 27--42.


\bibitem{guivarch_croissance_1973}
Y.~Guivarc'h, \emph{Croissance polynomiale et p\'eriodes des fonctions
  harmoniques}, Bull. Soc. Math. France \textbf{101} (1973), 333--379.

\bibitem{Helg} S. Helgason,
                 \emph{Groups and geometric analysis},
                 Academic Press, Orlando, FL, 1984.

\bibitem{hebisch_multiplier_1993}
W. Hebisch, \emph{Multiplier theorem on generalized {H}eisenberg groups},
Colloq. Math. \textbf{65} (1993), 231--239.

\bibitem{He}
W. Hebisch,
\emph{Functional calculus for slowly decaying kernels},
preprint (1995),
\texttt{http://www.math.uni.wroc.pl/{\textasciitilde}hebisch/}.



\bibitem{HMM_OU}
W. Hebisch, G. Mauceri, and S. Meda,
\emph{Holomorphy of spectral multipliers of the Ornstein-Uhlenbeck operator},
J. Funct. Anal. \textbf{210} (2004), 101--124.

\bibitem{HMM}
W. Hebisch, G. Mauceri, and S. Meda,
\emph{Spectral multipliers for sub-Laplacians with drift on Lie groups},
Math. Z. \textbf{251} (2005), 899--927.

\bibitem{HS} W. Hebisch and T. Steger,
                 \emph{Multipliers and singular integrals on expo\-nen\-tial growth groups},
                 Math. Z. \textbf{245} (2003) 37--61.

\bibitem{hewitt_abstract_1979}
E.~Hewitt and K.~A. Ross, \emph{Abstract harmonic analysis. {V}ol. {I}}, second
  ed., Springer, Berlin, 1979.



\bibitem{ionescu_singular_2002}
A. D. Ionescu, \emph{Singular integrals on symmetric spaces of real rank one},
Duke Math. J. \textbf{114} (2002), 101--122.

\bibitem{ionescu_singular_2003}
A. D. Ionescu, \emph{Singular integrals on symmetric spaces. II},
Trans. Amer. Math. Soc. \textbf{355} (2003), 3359--3378.

\bibitem{LS} H.-Q. Li and P. Sj\"ogren, \emph{Sharp endpoint estimates for some operators associated with the Laplacian with drift in Euclidean space}, preprint (2017), \texttt{arXiv:1701.04936}.

\bibitem{LSW} H.-Q. Li, P. Sj\"ogren, and Y. Wu, \emph{Weak type $(1,1)$ of some operators for the Laplacian with drift}, Math. Z. \textbf{282} (2016), 623--633. 

\bibitem{LM}
N.~Lohou{\'e} and S.~Mustapha, 
\emph{Sur les transform\'ees de {R}iesz dans le cas du {L}aplacien avec drift}, Trans. Amer. Math. Soc. \textbf{356} (2004), 2139--2147.


\bibitem{martini_algebras_2010}
A. Martini,
\emph{Algebras of differential operators on Lie groups and spectral multipliers}, Tesi di Perfezionamento (Ph.D. Thesis), Scuola Normale Superiore (Pisa), 2010, \texttt{arXiv:1007.1119}.



\bibitem{martini_spectral_2011}
A.~Martini, \emph{{Spectral theory for commutative algebras of differential
  operators on Lie groups}}, J. Funct. Anal. \textbf{260} (2011),
  2767--2814.
	
\bibitem{martini_analysis_2012}
A.~Martini, \emph{Analysis of joint spectral multipliers on {L}ie groups of polynomial growth},
Ann. Inst. Fourier (Grenoble) \textbf{62} (2012), 1215--1263.

\bibitem{M}
A. Martini,
\emph{Joint functional calculi and a sharp multiplier theorem for the Kohn Laplacian on spheres},
Math. Z.
(to appear),
\texttt{doi:10.1007/s00209-016-1813-8}.


\bibitem{martini_heisenbergreiter}
A. Martini, \emph{Spectral multipliers on Heisenberg--Reiter and related groups},
  Ann. Mat. Pura Appl. \textbf{194} (2015), 1135--1155.

\bibitem{MMue}
A. Martini and D. M\"uller,
Spectral multiplier theorems of Euclidean type on new classes of $2$-step stratified groups,
\textit{Proc. Lond. Math. Soc.} \textbf{109} (2014), 1229--1263.

\bibitem{MMue2}
A. Martini and D. M\"uller,
Spectral multipliers on $2$-step groups: topological versus homogeneous dimension,
\textit{Geom. Funct. Anal.} \textbf{26} (2016), 680--702.


\bibitem{MOV}
A. Martini, A. Ottazzi, and M. Vallarino, 
\emph{Spectral multipliers for sub-Laplacians on solvable extensions of stratified groups}, 
J. Analyse Math. (to appear), \texttt{arXiv:1504.03862}.




\bibitem{mauceri_vectorvalued_1990}
G. Mauceri and S. Meda, \emph{Vector-valued multipliers on stratified groups},
Rev. Mat. Iberoam. \textbf{6} (1990), 141--154.


\bibitem{MMV_preprint}
G. Mauceri, S. Meda, and M. Vallarino, \emph{Endpoint results for spherical multipliers on noncompact symmetric spaces}, preprint.


\bibitem{meda_weak_2010}
S. Meda and M. Vallarino, \emph{Weak type estimates for spherical multipliers on noncompact symmetric spaces},
Trans. Amer. Math. Soc. \textbf{362} (2010), 2993--3026.

\bibitem{MV} S. Meda and S. Volpi,
\emph{Spaces of Goldberg type on certain measured metric spaces}, 
Ann. Mat. Pura Appl. (to appear),
\texttt{doi:10.1007/s10231-016-0603-6}.

\bibitem{melrose_propagation_1986}
R.~Melrose, \emph{Propagation for the wave group of a positive subelliptic
  second-order differential operator}, in: Hyperbolic equations and related topics
  ({K}atata/{K}yoto, 1984), Academic Press, Boston, MA, 1986, pp.~181--192.

\bibitem{Montgomery}
    {R. Montgomery},
\emph{A tour of sub-Riemannian geometries, their geodesics and
              applications},
 {American Mathematical Society, Providence, RI},
     {2002}.



\bibitem{MueS}
D. M{\"u}ller and E. M. Stein, \emph{On spectral multipliers for {H}eisenberg and
  related groups}, J. Math. Pures Appl. \textbf{73} (1994), 413--440.

\bibitem{OV}  A. Ottazzi and M. Vallarino, \emph{Spectral multipliers for Laplacians with drift on Damek--Ricci spaces}, Math. Nachr. \textbf{287} (2014), 1837--1847.




\bibitem{stanton_expansions_1978}
R. J. Stanton and P. A. Tomas, \emph{Expansions for spherical functions on noncompact symmetric spaces},
Acta Math. \textbf{140} (1978), 251--276.



\bibitem{taylor}
M. E. Taylor, \emph{$L^p$-estimates on functions of the Laplace operator},
Duke Math. J. \textbf{58} (1989), 773--793.


\bibitem{tessera}
R. Tessera, \emph{Volume of spheres in doubling metric measured spaces and in groups of polynomial growth},
Bull. Soc. Math. France \textbf{135} (2007), 47--64.


\bibitem{V2} M. Vallarino, \emph{Spectral multipliers on Damek--Ricci spaces}, J. Lie Theory \textbf{17} (2007), 163--189.






\bibitem{varopoulos_analysis_1992}
N.~T. Varopoulos, L.~Saloff-Coste, and T.~Coulhon, \emph{Analysis and geometry
  on groups}, Cambridge University Press, Cambridge, 1992.


\end{thebibliography}
\end{document}